\documentclass[11pt,reqno]{amsart}
\usepackage[T1]{fontenc}
\usepackage[utf8]{inputenc}
\usepackage{mathpazo}
\usepackage{amsmath,amssymb,mathtools}
\usepackage{microtype}
\usepackage[margin=1in]{geometry}
\usepackage{graphicx}
\usepackage{enumitem}
\usepackage{xcolor}
\usepackage[colorlinks=true,linkcolor=blue!45!black,citecolor=blue!45!black,
  urlcolor=blue!45!black,bookmarksnumbered=true]{hyperref}
\hypersetup{pdftitle={Two degenerate intersections of Clifford tori in S2 x S2 and CP2},
  pdfauthor={Wenmin Gong},
  pdfsubject={An explicit smooth Hamiltonian construction}}
\numberwithin{equation}{section}
\newtheorem{theorem}{Theorem}[section]
\newtheorem{proposition}[theorem]{Proposition}
\newtheorem{lemma}[theorem]{Lemma}
\newtheorem{corollary}[theorem]{Corollary}
\theoremstyle{definition}

\theoremstyle{remark}
\newtheorem{remark}[theorem]{Remark}
\newcommand{\R}{\mathbb R}

\newcommand{\CP}{\mathbb{CP}}
\newcommand{\F}{\mathbb F}
\newcommand{\Ham}{\operatorname{Ham}}
\newcommand{\Crit}{\operatorname{Crit}}
\newcommand{\supp}{\operatorname{supp}}
\newcommand{\Rot}{\operatorname{Rot}}
\newcommand{\dd}{\,\mathrm d}
\newcommand{\id}{\operatorname{id}}
\newcommand{\eps}{\varepsilon}

\newcommand{\grad}{\nabla}
\setlist[enumerate]{label=(\roman*),leftmargin=2em}
\allowdisplaybreaks[2]

\title[Two intersections of Clifford tori]{Two degenerate intersections of Clifford tori\\
in $S^2\times S^2$ and $\mathbb{CP}^2$}
\author{Wenmin Gong}
\date{}

\address{Laboratory of Mathematics and Complex Systems (Ministry
of Education), School of Mathematical Sciences, Beijing Normal University,
	Beijing, 100875, China}

\email{ wmgong@bnu.edu.cn}

\begin{document}
\begin{abstract}
We construct smooth time-periodic Hamiltonians on the monotone manifolds
$S^2\times S^2$ and $\mathbb{CP}^2$ for which the Clifford torus meets its
Hamiltonian image in exactly two isolated degenerate points.  The product example disproves 
the degenerate Arnol’d–Givental conjecture.
\end{abstract}
\maketitle

\section{Introduction and main results}

The Lagrangian Arnol'd problem asks for lower bounds on
\(\#(L\cap\phi(L))\), where \(L\) is a closed Lagrangian submanifold and
\(\phi\) is Hamiltonian; see \cite{Arnold65,Arnold86} for the classical
formulation.  For transverse intersections, Floer theory provides
homological lower bounds under the classical hypotheses \cite{Floer88}.  For a
real Lagrangian, namely the fixed point set of an anti-symplectic involution,
the homological Arnol'd--Givental conjecture predicts the corresponding
Betti-number bound.  An early degenerate result is Givental's theorem that
every Hamiltonian image of \(\mathbb{RP}^n\subset\CP^n\) meets
\(\mathbb{RP}^n\) in at least \(n+1\) points, with no transversality
assumption; thus the full homological Arnol'd--Givental bound holds in this
model case \cite{Givental90}.  Chang--Jiang obtained a related cuplength
estimate for \((\CP^n,\mathbb{RP}^n)\), and Lu extended the degenerate
intersection bound to \((\CP^n(\mathbf q),\mathbb{RP}^n(\mathbf q))\) when
all weights are odd \cite{ChangJiang,LuWeighted}.  On the transverse side, Oh
established the Arnol'd--Givental inequality for real forms of compact
Hermitian symmetric spaces when the minimal Maslov number is at least two
\cite{OhAG92,OhAG95}, while Frauenfelder proved the conjecture for a class of
real Lagrangians in Marsden--Weinstein quotients using moment Floer homology
\cite{Frauenfelder}.
The transverse Arnol'd--Givental conjecture has recently been proved in full
generality \cite[Theorem~A]{BSWX}.  

The nontransverse problem is different.  A single isolated degenerate
intersection may split into several transverse intersections under an
arbitrarily small perturbation, so a transverse Floer complex does not by
itself bound the number of geometric intersections before perturbation.
The degenerate homological Arnol'd--Givental conjecture was proved in several settings: for zero
sections and exact cotangent-bundle problems \cite{Hofer85,LS}, under
relative asphericity or weak exactness \cite{Floer89,Hofer88}, and below
bubbling-energy thresholds \cite{Chekanov96,Chekanov98,Liu}.  In the real
setting, \cite{LuAG} proves a cup-length estimate under a rationality-index
bound on the Hofer norm.  More recently, Lagrangian
Lusternik--Schnirelmann methods and quantum cap products have produced uniform
degenerate lower bounds for several monotone examples \cite{Gong}.

We first consider the equal-area product.  Let
$S^2=\{\mathbf x\in\R^3:|\mathbf x|=1\}$ be the unit sphere.
We equip $S^2$ with the \emph{standard symplectic form} defined in
\cite[\S1.3, Example on $S^2$]{CdS}: for
$\mathbf x\in S^2$ and $u,v\in T_{\mathbf x}S^2=\mathbf x^\perp$,
\begin{equation}\label{eq:sphere-area-form}
 (\omega_{S^2})_{\mathbf x}(u,v)
 =\langle\mathbf x,u\times v\rangle.
\end{equation}
Here $\langle\cdot,\cdot\rangle$ (also denoted by $\cdot$ below) is the
Euclidean inner product on $\R^3$, and $\times$ is its standard cross
product in $\mathbb R^3$, with the usual orientation.  Thus, if
$\mathbf x=(x,y,z)$ and $\mathrm dx,\mathrm dy,\mathrm dz$ are the
restrictions of the ambient coordinate one-forms to $S^2$, the same
form is explicitly
\[
 \omega_{S^2}
 =x\,\mathrm dy\wedge\mathrm dz
  +y\,\mathrm dz\wedge\mathrm dx
  +z\,\mathrm dx\wedge\mathrm dy,
 \qquad \int_{S^2}\omega_{S^2}=4\pi.
\]

Let \(\pi_i:S^2\times S^2\to S^2\) be the two coordinate projections and set
\begin{equation}\label{eq:ambient}
 (M,\omega)=
 (S^2\times S^2,\pi_1^*\omega_{S^2}+\pi_2^*\omega_{S^2}).
\end{equation}
We write \(j=(j_x,j_y,j_z)\) and \(k=(k_x,k_y,k_z)\) for the two sphere
coordinates.  The product Clifford torus is
\begin{equation}\label{eq:L}
 L=\{(j,k)\in M:j_z=k_z=0\}\cong S^1\times S^1
\end{equation}
which is a Lagrangian submanifold in $S^2\times S^2$.  We use throughout the Hamiltonian convention
\begin{equation}\label{eq:ham-convention}
 \iota_{X_{H_t}}\omega=\mathrm dH_t,\qquad
 \frac{\mathrm d}{\mathrm dt}\varphi_H^t
   =X_{H_t}\circ\varphi_H^t,\qquad \varphi_H^0=\id.
\end{equation}
A Hamiltonian diffeomorphism is a map \(\varphi_H^1\); the group of all such
maps is denoted by \(\Ham(M,\omega)\).  Two Lagrangians meet transversely at
\(x\) when their tangent spaces span \(T_xM\); otherwise the intersection is
nontransverse, or degenerate.

\begin{theorem}\label{thm:main}
There is \(H\in C^\infty(\R/\mathbb Z\times M,\R)\) such that
\[
 \#\bigl(L\cap\varphi_H^1(L)\bigr)=2.
\]
Both intersections are isolated and nontransverse.  In fact, the two tangent
planes coincide at each intersection.
\end{theorem}

Put \(r(x,y,z)=(x,y,-z)\).  Then \(r\times r\) is anti-symplectic for  $(M,\omega)$ and
\(\operatorname{Fix}(r\times r)=L\).  Thus \(L\) is real.  Here
\(\F_2=\mathbb Z/2\mathbb Z\) denotes the field with two elements.  The
transverse Arnol'd--Givental theorem gives
\begin{equation}\label{eq:AG}
 L\pitchfork\phi(L)\quad\Longrightarrow\quad
 \#(L\cap\phi(L))\geq
 \dim_{\F_2}H_*(L;\F_2)=4
\end{equation}
for every Hamiltonian diffeomorphism \(\phi\) \cite[Theorem~A]{BSWX}.
Accordingly, Theorem~\ref{thm:main} should be read as a counterexample to the
 degenerate homological Arnol'd--Givental conjecture~\cite[\S 11.3]{McS}.

The second main result concerns the standard Clifford torus
\begin{equation}\label{eq:CP2-Clifford-intro}
 T_{\mathrm{Cl}}=
 \{[z_0:z_1:z_2]\in\CP^2:|z_0|=|z_1|=|z_2|\}.
\end{equation}
We normalize the Fubini--Study form by
\(\int_{\CP^1}\omega_{\mathrm{FS}}=2\pi\).

\begin{theorem}\label{thm:cp2}
There is a smooth time-periodic Hamiltonian
\(H_{\CP^2}:\R/\mathbb Z\times\CP^2\to\R\) such that
\[
 \#\bigl(T_{\mathrm{Cl}}\cap
 \varphi_{H_{\CP^2}}^1(T_{\mathrm{Cl}})\bigr)=2.
\]
The two intersections are isolated, and the tangent planes of the two
Lagrangians coincide at each of them.
\end{theorem}

Theorem~\ref{thm:cp2} is \emph{not} an Arnol'd--Givental example since
the Clifford torus in \(\CP^2\) cannot be a component of the fixed point set
of an anti-symplectic involution.  Indeed, by the Smith inequality, the total
\(\F_2\)-Betti number of the fixed set of a smooth involution is at most that
of the ambient manifold, but here 
\[
 \dim_{\F_2}H_*(T^2;\F_2)=4>3
 =\dim_{\F_2}H_*(\CP^2;\F_2).
\]
See \cite[\S1]{Brendel} for this obstruction and its relation to real
Lagrangians. 


Both results attain known uniform lower bounds.  For the product torus
\eqref{eq:L}, \cite[Theorem~7 and Example~15]{Gong} gives
\(\#(L\cap\phi(L))\ge2\) for every \(\phi\in\Ham(M,\omega)\).  For the
Clifford torus in \(\CP^n\), \cite[Theorem~2]{Gong} gives the same lower
bound.  Consequently:

\begin{corollary}\label{cor:sharp}
For both pairs
\[
 (S^2\times S^2,L),\qquad (\CP^2,T_{\mathrm{Cl}}),
\]
the minimum, over Hamiltonian diffeomorphisms, of the number of geometric
self-intersections with a Hamiltonian image is equal to two.
\end{corollary}

\begin{proof}
The cited lower bounds give two intersections or more for every
Hamiltonian image of either torus.  In the product case, the normalization
in \cite[Example~15]{Gong} differs from \eqref{eq:ambient} by a common
positive factor: replacing \(\omega\) and \(H\) by \(\lambda\omega\)
and \(\lambda H\) leaves the Hamiltonian vector field unchanged under
\eqref{eq:ham-convention}.  The matching upper bounds are provided by
Theorems~\ref{thm:main} and \ref{thm:cp2}.
\end{proof}

There is also a cup-length consequence.  Both Clifford tori above are
monotone, and the lower bounds just cited imply their nondisplaceability; see
also the discussion of these standard examples in \cite{Gong}.  The
cup-length convention used there gives
\(\operatorname{cl}_{\mathbb K}(T^2)=3\) for every field \(\mathbb K\):
there are two positive-degree factors in a nonzero cup product on
\(T^2\), and this convention adds one.  Hence either main theorem
produces two isolated intersections, fewer than this cup-length.  In particular, the
unrestricted monotone Lagrangian cup-length statement formulated in
\cite[Conjecture~2]{Gong} does not hold without an additional hypothesis.

\subsection*{Idea of the proof}

The decisive local mechanism in both proofs is classical rather than a new
singularity-theoretic phenomenon.  The cubic
\begin{equation}\label{eq:cubic-intro}
 c+t(t^2-s^2)
\end{equation}
is, after a linear rescaling, the elliptic-umbilic catastrophe in Thom's
theory and the real \(D_4^-\) simple singularity in Arnol'd's classification
\cite{Thom72,Arnold72Sing}.  Its standard unfolding already exhibits the
coalescence of two equal-level Morse saddles into one degenerate cubic
critical point; for example,
\[
 x^3-3xy^2+\lambda x,\qquad \lambda>0,
\]
has two Morse saddles of the same critical value joined by a regular
zero-level arc, and they merge into \(x^3-3xy^2\) as \(\lambda\to0^+\).
Thus the prescribed cubic degeneration belongs to the classical
singularity theory of Thom and Arnol'd.  A complementary classical precedent
is Takens' procedure for combining critical points on a connected critical
level \cite[Theorem~(2.7) and the proof of Proposition~(2.9)]{Takens68}.
Takens' construction localizes the modification and controls the number of
remaining critical points, but does not prescribe a smooth \(D_4^-\) normal
form for the resulting point.  In the present symplectic problem, we require
both a \emph{relative} merger with precisely the cubic normal form
\eqref{eq:cubic-intro} and a compactly supported Hamiltonian realization
without additional Lagrangian intersections.  We first locate each saddle
pair on an embedded regular level arc; Section~5 then proves the required
relative saddle-merging lemma.  A related recent implementation appears in
\cite[\S3.1 and Lemma~3.1]{Jiao}.

For Theorem~\ref{thm:main}, we first apply a radial Hamiltonian depending on
\(r=|j+k|\).  Its time-one map \(\Phi_\rho\) rotates the two sphere factors
simultaneously about the axis \((j+k)/|j+k|\), with an angle that is nearly
\(\pi/2\) except in a transition region of size \(O(\rho)\).  Writing
\(L_\rho=\Phi_\rho(L)\) and rotating the reference torus to
\(L_*=\mathcal R(L)\), we obtain exactly four transverse intersections
\(L_\rho\cap L_*\), two near each point \(P_\sigma=(\sigma N_0,-\sigma N_0)\), where \(N_0=(0,0,1)\) and \(\sigma\in\{\pm1\}\).
The auxiliary rotation is crucial to our construction: it places the reference torus through the two polar neighborhoods and produces the indefinite quadratic term in the limiting potential that gives rise to the required saddle pairs. 
Near \(P_\sigma\), the cotangent-bundle coordinates of Section~2 express the
two Lagrangians locally as Lagrangian graphs over the momentum coordinates
\[
 q=\nabla F_\sigma(p),\qquad q=\nabla A_\rho(p).
\]
The four intersections occur on the shrinking scale \(|p|=O(\rho)\).  After
setting \(p=\rho v\) and rescaling the potentials by \(\rho^{-2}\), the
potential differences converge on a fixed annulus to the explicit planar
functions
\[
 f_{0,\sigma}(v)=h(|v|)-\frac{\sigma}{2}v_1v_2.
\]
The key role of this rescaling is to reduce the shrinking intersection problem to a fixed annulus with a smooth limiting potential, allowing the saddle merger to be constructed at the limiting scale and then transferred back to the original problem by an exact compactly supported correction.
Each limiting function has exactly two nondegenerate saddles, and
Lemma~\ref{lem:arc} shows that they lie on one regular level arc.  We then
apply the classical \(D_4^-\) saddle merger in the relative form of
Lemma~\ref{lem:merge}, and transfer this modification from the limiting
function back to the finite-\(\rho\) graph by an exact compactly supported
correction.  The corresponding Hamiltonian shear \(\Psi_\rho\) replaces the
two intersections near each \(P_\sigma\) by one isolated cubic intersection.
Because the Hessian of \eqref{eq:cubic-intro} vanishes at its critical point,
the two Lagrangian tangent planes coincide there.  Finally
\(\mathcal R^{-1}\) returns the reference torus from \(L_*\) to \(L\), and
the three Hamiltonian stages are placed in disjoint time intervals to produce
the periodic Hamiltonian in Theorem~\ref{thm:main}.

The proof of Theorem~\ref{thm:cp2} uses the same local singularity-theoretic
move but a different global preparation.  A specific projective unitary
transformation \(\Phi_U\) sends the Clifford torus to a Lagrangian having
exactly four transverse intersections with \(T_{\mathrm{Cl}}\).  A finite
symmetry organizes these four points into two pairs.  In toric action--angle
coordinates, each pair lies on an explicit arc \(\Gamma_\sigma\) along which
the Liouville one-form vanishes.  Consequently a neighborhood of each arc is
a single exact position graph \(p=\nabla f_\sigma(q)\), the two endpoints are
Morse saddles of the same critical value, and the arc between them is a
regular level arc.  Lemma~\ref{lem:merge} can therefore be applied directly
to each pair.  Two compactly supported Hamiltonian graph shears replace the
four transverse intersections by two isolated \(D_4^-\)-type cubic
intersections, again with coincident tangent planes.  By a graph shear
we mean a local Hamiltonian translation such as
\((q,p)\mapsto(q,p+\nabla h(q))\), which changes a position-graph
potential \(f\) to \(f+h\); the dual momentum-graph shear is used
for Theorem~\ref{thm:main}.  Combining the
projective unitary step and the local shears in disjoint time intervals gives
the Hamiltonian of Theorem~\ref{thm:cp2}.

\section*{Acknowledgments}
I would like to thank Professors Guangcun Lu and Jinxin Xue for their sustained support and encouragement of my research on the degenerate Arnold (-Givental) conjecture over the past six years. Over the past period of time, I had always believed that the degenerate Arnol'd(-Givental) conjecture was correct, and I even spent a great deal of time trying to prove that any Hamiltonian map $S^2\times S^2$ has at least 3 fixed points. Jiao's paper shattered this belief of mine. My previous work on the degenerate homological Arnol'd-Givental conjecture, combined with AI, led to this paper.

\section*{AI disclosure}
Large language model tools were used to assist in the preparation of this article.
The author takes full responsibility for its
mathematical content and for any errors that may remain.

\section{Canonical coordinates and local Lagrangian generating functions}
A canonical coordinate system \((q,p)\), with
\(q=(q_1,q_2)\) and \(p=(p_1,p_2)\), means one in which
\(\omega=\sum_{i=1}^2\mathrm dq_i\wedge\mathrm dp_i\).  For a momentum
graph \(q=\nabla A(p)\), the scalar function \(A\) will be called its
potential; the corresponding elementary generating-function facts are
recorded below.

We first establish the coordinate formulas used in the construction.
Using the round metric, we identify
\[
 T^*S^2=\{(n,\eta)\in\R^3\times\R^3:
 |n|=1,\ \eta\cdot n=0\},
\]
where the vector $\eta\in T_nS^2$ represents the covector
$\eta^\flat\in T_n^*S^2$ defined by
$\eta^\flat(v)=\eta\cdot v$.
Let $\pi:T^*S^2\to S^2$ be the bundle projection and
$\iota:S^2\hookrightarrow\R^3$ the standard inclusion.
To distinguish functions on the cotangent bundle from their pullbacks,
define the $\R^3$-valued maps
\[
 \bar n=\iota\circ\pi,\qquad
 \bar\eta:T^*S^2\longrightarrow\R^3,\quad
 \bar\eta(n,\eta)=\eta.
\]
Thus $\bar n(n,\eta)=n$.  If $\lambda$ denotes the tautological
one-form, then, by its defining property,
\[
 \lambda_{(n,\eta)}(\dot n,\dot\eta)
 =\eta^\flat\bigl(\mathrm d\pi_{(n,\eta)}(\dot n,\dot\eta)\bigr)
 =\eta\cdot\dot n.
\]
Consequently,
\begin{equation}\label{eq:canonical-cotangent-form}
 \lambda=\bar\eta\cdot\mathrm d\bar n,\qquad
 \omega_{\mathrm{can}}=-\mathrm d\lambda
 =-\mathrm d(\bar\eta\cdot\mathrm d\bar n).
\end{equation}
This is the standard tautological construction of the canonical
symplectic form on a cotangent bundle; see
\cite[\S2.2--\S2.3]{CdS}.  The symplectic identification of the
complement of the diagonal in the equal-area product $S^2\times S^2$
with a disk cotangent bundle of $S^2$ is standard and can be viewed
through symplectic-cut compactification; see
\cite[Proposition~4.4]{AudinCut} and \cite[\S2.1.1]{LiWu}.
For the later calculations we need a particular explicit realization,
which we verify below; compare also \cite[\S2.2]{Jiao}.
Here and below, $j,k:M\to\R^3$ also denote the two ambient coordinate
maps of the sphere factors.

\begin{lemma}\label{lem:coordinates}
Let $\Delta=\{(j,k)\in M:j=k\}$ be the diagonal and let
\[
 D_2T^*S^2=\{(n,\eta)\in T^*S^2:|\eta|<2\}
\]
be the open cotangent disk bundle of radius $2$, with the fiber norm
induced by the round metric. On $M\setminus\Delta$ define
\[
 S(j,k)=j+k,\qquad r(j,k)=|S(j,k)|,
\]
and the $\R^3$-valued functions
\[
 N(j,k)=\frac{j-k}{|j-k|},\qquad
 E(j,k)=S(j,k)\times N(j,k).
\]
Then
\begin{equation}\label{eq:n-eta}
 \begin{aligned}
 \Psi:M\setminus\Delta&\longrightarrow D_2T^*S^2,\\
 (j,k)&\longmapsto\bigl(N(j,k),E(j,k)\bigr)
 \end{aligned}
\end{equation}
is a symplectomorphism from $(M\setminus\Delta,\omega)$ to
$(D_2T^*S^2,\omega_{\mathrm{can}})$. Its inverse, written in terms of
a point $(n,\eta)$ in the \emph{target} cotangent bundle, is
\begin{equation}\label{eq:inverse}
 j=a n+\tfrac12S,\qquad k=-a n+\tfrac12S,\qquad
 S=n\times\eta,\quad r=|\eta|,\quad
 a=\sqrt{1-r^2/4}.
\end{equation}

On each of the two open regions $N_z>0$ and $N_z<0$ in $M\setminus\Delta$,
introduce the local base coordinates $q=(q_1,q_2)=(N_x,N_y)$.
Let $p=(p_1,p_2)\in\R^2$ be the corresponding momentum coordinates,
characterized by the \emph{pulled-back} tautological one-form
\[
 \Psi^*\lambda=p_1\,\mathrm dq_1+p_2\,\mathrm dq_2.
\]
All $q_i,p_i$ in this statement are therefore local functions on
$M\setminus\Delta$. Then
\begin{align}
 N_z&=\sigma\sqrt{1-|q|^2},\qquad
 \sigma=\begin{cases}1,&N_z>0,\\-1,&N_z<0,\end{cases}
 \label{eq:nz}\\
 S&=(-N_zp_2,N_zp_1,q_1p_2-q_2p_1),\label{eq:S-p}\\
 r^2&=|p|^2-(q\cdot p)^2,\label{eq:canonical}\\
 \omega&=\mathrm dq_1\wedge\mathrm dp_1+
        \mathrm dq_2\wedge\mathrm dp_2.\notag
\end{align}
\end{lemma}
\begin{proof}
We first verify that $\Psi$ is a diffeomorphism. Since $j,k$ are
unit vectors,
\[
 S\cdot(j-k)=|j|^2-|k|^2=0,\qquad
 |S|^2+|j-k|^2=4.
\]
Thus $N\in S^2$, $S\perp N$, and $E=S\times N$ belongs to
$T_NS^2\cong T_N^*S^2$, with $|E|=r<2$.
The vector triple-product identity gives
$N\times E=N\times(S\times N)=S$.
Conversely, start with $(n,\eta)\in D_2T^*S^2$, put
$S=n\times\eta$ and $a=\sqrt{1-|\eta|^2/4}>0$, and define $j,k$
by \eqref{eq:inverse}. Since $S\perp n$ and $|S|=|\eta|$,
\[
 |j|^2=|k|^2=a^2+\tfrac14|S|^2=1,
 \qquad j-k=2an\neq0.
\]
These formulas are smooth and mutually inverse, proving that $\Psi$
is a diffeomorphism.

Before verifying the symplectic identity, we make the pullback of the
one-form precise. By definition of $N,E$,
\[
 \bar n\circ\Psi=N,\qquad \bar\eta\circ\Psi=E.
\]
Thus the standard pullback and exterior-differentiation rules give
\begin{equation}\label{eq:pullback-lambda}
 \begin{aligned}
 \Psi^*\lambda
 &=\Psi^*(\bar\eta\cdot\mathrm d\bar n)
   =E\cdot\mathrm dN\\
 &=\frac{(S\times(j-k))\cdot\mathrm d(j-k)}{|j-k|^2}\\
 &=\frac{(j\times k)\cdot(\mathrm dk-\mathrm dj)}{1-j\cdot k}.
 \end{aligned}
\end{equation}
Here $\mathrm dj,\mathrm dk$ are the $\R^3$-valued differentials
of the two coordinate maps on $M$. In particular, every expression
in \eqref{eq:pullback-lambda} is a one-form on $M\setminus\Delta$.

We next show directly that $\Psi$ preserves the symplectic forms.
Let $SO(3)$ denote the group of orientation-preserving linear
isometries of $\R^3$. It acts diagonally on $M$ by
$R\cdot(j,k)=(Rj,Rk)$ and, via the round metric, by cotangent
lifts on $T^*S^2$ by $R\cdot(n,\eta)=(Rn,R\eta)$.
Since rotations preserve lengths and cross products,
\[
 \Psi(Rj,Rk)=R\cdot\Psi(j,k)\qquad(R\in SO(3));
\]
that is, $\Psi$ is $SO(3)$-equivariant.

Fix $b\in\R^3$. Let $R_t$ be the one-parameter family of rotations
whose infinitesimal action on $v\in\R^3$ is
$\left.\frac{\mathrm d}{\mathrm dt}\right|_{t=0}R_tv=b\times v$.
The corresponding infinitesimal generators on $M$ and $T^*S^2$ are
respectively
\[
 X_b(j,k)=(b\times j,b\times k),\qquad
 Y_b(n,\eta)=(b\times n,b\times\eta).
\]
Differentiating $\Psi(R_tj,R_tk)=R_t\cdot\Psi(j,k)$ at $t=0$
gives the precise relation
\begin{equation}\label{eq:Psi-infinitesimal}
 \mathrm d\Psi_{(j,k)}\bigl(X_b(j,k)\bigr)
 =Y_b\bigl(\Psi(j,k)\bigr).
\end{equation}
Using the explicitly prescribed sphere form
\eqref{eq:sphere-area-form}, for every $v\in T_jS^2$ we obtain
\[
 \begin{aligned}
 (\omega_{S^2})_j(b\times j,v)
 &=j\cdot\bigl((b\times j)\times v\bigr)\\
 &=j\cdot\bigl(j(b\cdot v)-b(j\cdot v)\bigr)
 =b\cdot v,
 \end{aligned}
\]
where $j\cdot v=0$ and $|j|=1$ were used. The same calculation
holds on the second sphere. Consequently, with the Hamiltonian
sign convention \eqref{eq:ham-convention},
\[
 \iota_{X_b}\omega
 =\mathrm d(b\cdot j)+\mathrm d(b\cdot k)
 =\mathrm d(b\cdot S).
\]

The cotangent-lifted rotations preserve the tautological one-form
$\lambda$: this follows from
$(R\bar\eta)\cdot\mathrm d(R\bar n)
=\bar\eta\cdot\mathrm d\bar n$ for a fixed $R\in SO(3)$. 
Hence $\mathcal{L}_{Y_b}\lambda=0$. Then by Cartan's formula, together with
$\omega_{\mathrm{can}}=-\mathrm d\lambda$, yields
\[
 \begin{aligned}
 \iota_{Y_b}\omega_{\mathrm{can}}
 &=\mathrm d\bigl(\lambda(Y_b)\bigr)\\
 &=\mathrm d\bigl(\bar\eta\cdot(b\times\bar n)\bigr)\\
 &=\mathrm d\bigl(b\cdot(\bar n\times\bar\eta)\bigr).
 \end{aligned}
\]
Because $(\bar n\times\bar\eta)\circ\Psi=N\times E=S$,
equation \eqref{eq:Psi-infinitesimal} now implies
\[
 \iota_{X_b}\bigl(\Psi^*\omega_{\mathrm{can}}\bigr)
 =\Psi^*\bigl(\iota_{Y_b}\omega_{\mathrm{can}}\bigr)
 =\mathrm d(b\cdot S)
 =\iota_{X_b}\omega.
\]
Define the two-form $\beta=\Psi^*\omega_{\mathrm{can}}-\omega$
on $M\setminus\Delta$. We have shown that
$\iota_{X_b}\beta=0$ for every $b\in\R^3$.
If $j\times k\neq0$, the linear map $b\mapsto X_b(j,k)$
is injective: a vector satisfying $b\times j=b\times k=0$
must be zero because $j$ and $k$ are not parallel.
Thus the vectors $X_b(j,k)$ span a three-dimensional subspace
of the four-dimensional tangent space. Since $\beta$ vanishes
whenever one of its arguments belongs to this subspace and is
skew-symmetric, it vanishes identically at $(j,k)$.
The points with $j\times k\neq0$ are dense in $M\setminus\Delta$;
by continuity the same identity holds at $j=-k$.
Therefore
\[
\Psi^*\omega_{\mathrm{can}}
 =\pi_1^*\omega_{S^2}+\pi_2^*\omega_{S^2}=\omega.
\]
This proves that $\Psi$ is a symplectic diffeomorphism.

It remains to derive the explicit coordinate identities. On either
region where $N_z\neq0$, the unit-length condition
$N=(q_1,q_2,N_z)\in S^2$ gives \eqref{eq:nz} and
\[
 \mathrm dN_z
 =-\frac{q_1}{N_z}\,\mathrm dq_1
  -\frac{q_2}{N_z}\,\mathrm dq_2.
\]
Writing $E=(E_x,E_y,E_z)$, the pullback formula
\eqref{eq:pullback-lambda} yields
\[
 \begin{aligned}
 \Psi^*\lambda=E\cdot\mathrm dN
 &=E_x\,\mathrm dq_1+E_y\,\mathrm dq_2
   +E_z\,\mathrm dN_z\\
 &=\left(E_x-\frac{E_zq_1}{N_z}\right)\mathrm dq_1
  +\left(E_y-\frac{E_zq_2}{N_z}\right)\mathrm dq_2.
 \end{aligned}
\]
Comparing coefficients with the defining relation
$\Psi^*\lambda=p_1\,\mathrm dq_1+p_2\,\mathrm dq_2$ gives
\begin{equation}\label{eq:p-from-eta}
 p_1=E_x-\frac{E_zq_1}{N_z},\qquad
 p_2=E_y-\frac{E_zq_2}{N_z}.
\end{equation}
Since $E\cdot N=0$, these identities imply
\[
 \begin{aligned}
 0&=q_1E_x+q_2E_y+N_zE_z\\
  &=q\cdot p+\frac{E_z}{N_z}
    (q_1^2+q_2^2+N_z^2)
   =q\cdot p+\frac{E_z}{N_z}.
 \end{aligned}
\]
Thus $E_z=-N_z(q\cdot p)$; substituting back into
\eqref{eq:p-from-eta} gives
\begin{equation}\label{eq:eta-from-p}
 E=\bigl(p_1-q_1(q\cdot p),\ 
          p_2-q_2(q\cdot p),\ 
          -N_z(q\cdot p)\bigr).
\end{equation}
For an equivalent geometric expression, put
$\widetilde p=(p_1,p_2,0)\in\R^3$ and
$P=I-NN^T$, where $I$ is the $3\times3$ identity matrix and
$N^T$ denotes the transpose of the column vector $N$.
Then $Pv=v-(N\cdot v)N$ is the orthogonal projection of
$v\in\R^3$ onto $N^\perp=T_NS^2$, and
$E=P\widetilde p$.
Consequently,
\[
 S=N\times E=N\times\widetilde p
 =(-N_zp_2,N_zp_1,q_1p_2-q_2p_1),
\]
which proves \eqref{eq:S-p}. In particular, since $N_z\neq0$,
\[
 p_1=\frac{S_y}{N_z},\qquad
 p_2=-\frac{S_x}{N_z}.
\]
These formulas express the momentum coordinates in terms of
$(N,S)$ and will also be used in later calculations. Moreover,
\[
 \begin{aligned}
 r^2=|S|^2=|E|^2
 &=|\widetilde p|^2-(N\cdot\widetilde p)^2\\
 &=|p|^2-(q\cdot p)^2,
 \end{aligned}
\]
which proves the first equality in \eqref{eq:canonical}.
Finally, combining the symplectic identity
$\Psi^*\omega_{\mathrm{can}}=\omega$, proved above using
$SO(3)$-equivariance, with the pullback formula
\eqref{eq:pullback-lambda} and its local expression
$\Psi^*\lambda=p_1\,\mathrm dq_1+p_2\,\mathrm dq_2$,
we obtain
\[
 \begin{aligned}
 \omega
 &=\Psi^*\omega_{\mathrm{can}}
 =-\Psi^*(\mathrm d\lambda)
 =-\mathrm d(\Psi^*\lambda)\\
 &=-\mathrm d(p_1\,\mathrm dq_1+p_2\,\mathrm dq_2)\\
 &=\mathrm dq_1\wedge\mathrm dp_1+
   \mathrm dq_2\wedge\mathrm dp_2.
 \end{aligned}
\]
This completes the proof of the asserted Darboux
coordinate formulas.
\end{proof}

\begin{remark}
For brevity in the subsequent sections, we use the conventional symbols
$n=N$ and $\eta=E$ for these \emph{source-side pullback functions} on
$M\setminus\Delta$ whenever the context is clear. In particular,
$n_z=N_z$ in formulas involving the local Darboux coordinates.
\end{remark}

\begin{lemma}\label{lem:graph}
Let $(q,p)$ be local Darboux coordinates with
$\omega=\sum_{i=1}^2\mathrm dq_i\wedge\mathrm dp_i$.
Let $U\subset\R^2$ be open, and let
$A,F\in C^\infty(U)$ be such that the graphs
\[
\Gamma_A=\{(q,p):q=\nabla A(p),\ p\in U\},
\qquad
\Gamma_F=\{(q,p):q=\nabla F(p),\ p\in U\}
\]
lie in the coordinate domain.
Then $\Gamma_A$ and $\Gamma_F$ are Lagrangian
submanifolds. 
Their intersections within the chart correspond precisely to
critical points of $F-A$.
Such an intersection is transverse if and only if
$D^2(F-A)$ is invertible at its momentum.
If this Hessian vanishes, the two tangent
spaces coincide.
\end{lemma}

\begin{proof}
The graph $\Gamma_A$ is parametrized by
$p\mapsto(\nabla A(p),p)$. Pulling back
the symplectic form gives
\[
\sum_{i,\ell=1}^2
\frac{\partial^2 A}{\partial p_i\partial p_\ell}
\,\mathrm dp_\ell\wedge\mathrm dp_i=0
\]
by symmetry of the Hessian. Since the graph
has half the ambient dimension, it is Lagrangian.
The same argument applies to $\Gamma_F$.

An intersection occurs precisely when
$\nabla(F-A)(p)=0$.
At such a point, the tangent spaces are
\[
T\Gamma_A=\{(D^2A(p)\xi,\xi):\xi\in\R^2\},
\]
and
\[
T\Gamma_F=\{(D^2F(p)\xi,\xi):\xi\in\R^2\}.
\]
Their tangent-space intersection is parametrized by
$\ker D^2(F-A)(p)$.
Hence the intersection is transverse if and only if
$D^2(F-A)(p)$ is invertible.
If $D^2(F-A)(p)=0$, the two tangent spaces
coincide.
\end{proof}

\section{A simultaneous rotation with four intersections}

We specify all initial cutoff profiles. Let
\[
 u(x)=
 \begin{cases}e^{-1/x},&x>0,\\0,&x\leq0,\end{cases}.
\]
This standard flat function is used to construct smooth cutoff profiles that are exactly constant outside prescribed transition intervals; its derivatives of every order vanish at the origin. 
We set
\begin{equation}\label{eq:profiles}
 \kappa(R)=100\frac{u(7/4-R)}{u(7/4-R)+u(R-5/4)},\qquad
 m(R)=\frac{u(R-1/2)}{u(R-1/2)+u(1-R)}.
\end{equation}
The denominators never vanish. The vanishing of every derivative of
\(u\) at zero proves smoothness across the endpoints.
The function \(\kappa\) is \(100\) on \(R\leq5/4\), zero on
\(R\geq7/4\), and strictly decreasing between these values.
The function \(m\) is nondecreasing, zero on \(R\leq1/2\), and
one on \(R\geq1\).
For \(0<\rho<\pi/400\), define
\begin{align}
 \theta_\rho(r)&=
 m(r/\rho)\bigl(\pi/2-\rho\kappa(r/\rho)\bigr),\label{eq:theta}\\
 K_\rho(j,k)&=\int_0^{|j+k|}\theta_\rho(s)\dd s,\qquad
 \Phi_\rho=\varphi_{K_\rho}^1,\qquad L_\rho=\Phi_\rho(L).\label{eq:bend}
\end{align}
The angle \(\theta_\rho\) is nondecreasing and lies in \([0,\pi/2]\).
It is zero on \(r\leq\rho/2\) and \(\pi/2\) on
\(r\geq7\rho/4\).

\begin{lemma}\label{lem:rotation}
The function \(K_\rho\) is smooth on \(M\).
For \(r>0\), its time-one map $ \Phi_\rho$ rotates \(j\) and \(k\) simultaneously
about \(S/r\) through the angle \(\theta_\rho(r)\), and 
it is the identity on \(r\leq\rho/2\).
\end{lemma}
\begin{proof}
The function $\theta_\rho$ vanishes on \(r\leq\rho/2\); hence the nonsmoothness of
\(|j+k|\) at \(j+k=0\) causes no problem.
Away from the set $\{(j,k)|\;|j+k|<\rho/2\}$, \(K_\rho(j,k)\) is smooth, including at \(j=k\).
For \(b=S/r\) with $S=j+k$ and $r=|j+k|$, differentiation gives
\(\mathrm dK_\rho(\xi,\eta)=\theta_\rho(r)b\cdot(\xi+\eta)\) for any $\xi,\eta\in\mathbb R^3$.
Recall that the double cross product formula is $(b\times j)\times\xi=(b\cdot \xi)j-(j\cdot\xi)b$. 
By \eqref{eq:sphere-area-form}, the identity
\[
 j\cdot((b\times j)\times\xi)=b\cdot\xi,\qquad \xi\perp j,
\]
and its second-factor analogue show
\[
 X_{K_\rho}=\theta_\rho(r)(b\times j,b\times k).
\]
Along the flow of $X_{K_\rho}$, \(\dot S=\theta_\rho(r)b\times S=0\).
The axis and the angle speed are constant along each trajectory,
which gives the stated rotation.
\end{proof}

Set \(N_0=(0,0,1)\) and put
\begin{equation}\label{eq:target}
 P_\sigma=(\sigma N_0,-\sigma N_0),\quad \sigma=\pm1,\qquad
 L_*=\{j_y=0,\ k_x=0\}.
\end{equation}
Let \(e_x,e_y,e_z\) denote the standard coordinate vectors in \(\R^3\),
and let \(\Rot_{v,\theta}\) denote the orientation-preserving rotation about
the unit axis \(v\) through angle \(\theta\).  The torus \(L_*\) is the image
of \(L\) under
\begin{equation}\label{eq:Rot}
 \mathcal R=
 \Rot_{e_x,\pi/2}\times\Rot_{e_y,\pi/2}.
\end{equation}
These product rotations are Hamiltonian: a linear function
\(b\cdot j\) has Hamiltonian vector field \(b\times j\).

\begin{proposition}\label{prop:four}
For every sufficiently small \(\rho>0\),
\(L_\rho\cap L_*\) consists of exactly four points, two approaching each
\(P_\sigma\). At all four intersection points,
\[
r=|S|=|j+k|=r_\rho,
\]
where \(r_\rho\in(5\rho/4,7\rho/4)\) is the unique solution of
\begin{equation}\label{eq:root}
\rho\kappa(r/\rho)
=
\arctan\frac{r}{2\sqrt{1-r^2/2}}.
\end{equation}
Moreover,
\begin{equation}\label{eq:Rs-limit}
\frac{r_\rho}{\rho}\longrightarrow R_s
\qquad\text{as }\rho\to0^+,
\end{equation}
where
\[
\kappa(R_s)=\frac{R_s}{2},
\qquad
\frac54<R_s<\frac74.
\]
\end{proposition}
\begin{proof}
All points of \(L_\rho\) have \(S_z=0\), since this holds on \(L\)
and \(\Phi_\rho\) preserves \(S\).
An intersection therefore has
\[
 j=(a,0,z),\qquad k=(0,b,-z),\qquad
 a^2=b^2=r^2/2,\qquad z^2=1-r^2/2.
\]
Thus every point of \(L_\rho\cap L_*\) must be among these candidate
points in \(L_*\); in particular, \(r>\sqrt2\) is impossible.
 We now determine which of them actually lie in
\(L_\rho=\Phi_\rho(L)\). This is equivalent to requiring
\[
 \Phi_\rho^{-1}(j,k)\in L=\{j_z=k_z=0\}.
\]
Write
\[
 \Phi_\rho^{-1}(j,k)=(j',k'),\qquad d'=j'-k'.
\]
Since \(\Phi_\rho\) preserves \(S=j+k\), we have
\(j'+k'=S\), and the candidate points above satisfy \(S_z=0\). Hence
\[
 \Phi_\rho^{-1}(j,k)\in L
 \quad\Longleftrightarrow\quad
 j'_z=k'_z=0
 \quad\Longleftrightarrow\quad
 (d')_z=0.
\]
For \(r>0\), Lemma~\ref{lem:rotation} gives
\[
 d'=\Rot_{S/r,-\theta_\rho(r)}d,
 \qquad d=j-k.
\]
Thus, for \(r>0\), the intersection condition is exactly the vanishing
of the \(z\)-component of this inverse-rotated difference vector. The
case \(r=0\) is treated separately, since the axis \(S/r\) is then
undefined.

If \(r=0\), the only possibilities in \(L_*\) are \(P_\pm\).
Since \(r=0\leq\rho/2\), Lemma~\ref{lem:rotation} gives
\(\Phi_\rho^{-1}(P_\pm)=P_\pm\), while \(P_\pm\notin L\).
Hence \(P_\pm\notin L_\rho=\Phi_\rho(L)\).

If \(r=\sqrt2\), then
\[
 z=0,\qquad a^2=b^2=1,
\]
so
\[
 j=(a,0,0),\qquad k=(0,b,0),
\]
with \(a,b\in\{\pm1\}\). For sufficiently small \(\rho\),
\[
 \theta_\rho(r)=\frac{\pi}{2}.
\]
Set
\[
 S=j+k=(a,b,0),\qquad
 d=j-k=(a,-b,0),\qquad
 B=\frac{S}{r}.
\]
Since \(B\cdot d=0\), Rodrigues' formula gives
\[
 d'=\Rot_{B,-\pi/2}(d)=-B\times d.
\]
Hence
\[
 (d')_z
 =
 -\frac1r\bigl((a,b,0)\times(a,-b,0)\bigr)_z
 =
 \frac{2ab}{r}
 \neq0.
\]
By the criterion above, the point does not lie in \(L_\rho\). Hence
\(r=\sqrt2\) cannot occur at an intersection point of \(L_\rho\) and
\(L_*\).

For \(0<r<\sqrt2\), let
\[
 d=j-k=(a,-b,2z),\qquad B=\frac{S}{r}=\frac{(a,b,0)}{r}.
\]
Since \(B\cdot d=0\), Rodrigues' formula yields
\(d'= |d'|\cos\theta \cdot \frac{d}{|d|}-  |d'|\sin\theta\cdot\frac{B\times d}{|B\times d|}\). Hence, 
\[
 (d')_z
 =
 2z\cos\theta_\rho(r)
 +\frac{2ab}{r}\sin\theta_\rho(r).
\]
Thus the intersection criterion is
\[
 2z\cos\theta_\rho(r)
 +\frac{2ab}{r}\sin\theta_\rho(r)=0.
\]
At any solution with \(0<r<\sqrt2\), one necessarily has
\[
 0<\theta_\rho(r)<\pi/2.
\]
Indeed, \(\theta_\rho(r)=0\) would force \(z=0\), while
\(\theta_\rho(r)=\pi/2\) would force \(ab=0\), both impossible when
\(0<r<\sqrt2\). Hence the vanishing condition implies
\(\operatorname{sgn}z=-\operatorname{sgn}(ab)\). Since
\[
 |z|=\sqrt{1-r^2/2},\qquad |ab|=r^2/2,
\]
it reduces to
\[
 \sqrt{1-r^2/2}\,\cos\theta_\rho(r)
 =\frac r2\sin\theta_\rho(r).
\]
Therefore
\begin{equation}\label{eq:global-root}
 \pi/2-\theta_\rho(r)=
 \arctan\frac{r}{2\sqrt{1-r^2/2}}.
\end{equation}
The left side is nonincreasing and the right side strictly increasing.
At \(r=5\rho/4\), the left side equals \(100\rho\), exceeding the
right side for small \(\rho\); at \(r=7\rho/4\), the left side is
zero while the right side is positive. Thus the intermediate value
theorem gives a solution in \((5\rho/4,7\rho/4)\), and the opposite
monotonicities give uniqueness; they also exclude any solution outside
this interval. For \(r\in[5\rho/4,7\rho/4]\), we have \(r/\rho\ge1\), hence
\(m(r/\rho)=1\). Therefore, by the definition of \(\theta_\rho\),
\[
\frac{\pi}{2}-\theta_\rho(r)
=\rho\kappa(r/\rho).
\]
Thus \eqref{eq:global-root} becomes
\[
\rho\kappa(r/\rho)
=
\arctan\frac{r}{2\sqrt{1-r^2/2}}.
\]
Hence \eqref{eq:global-root} is exactly \eqref{eq:root}.
Each of the four sign choices
\(a,b=\pm r_\rho/\sqrt2\) gives one point; the sign of \(z\)
is already determined. The cases \(ab<0\) approach \(P_+\),
and the cases \(ab>0\) approach \(P_-\).

Set \(R_\rho=r_\rho/\rho\). Uniformly for
\(R\in[5/4,7/4]\), Taylor's formula gives
\begin{equation}\label{eq:root-expansion}
 \kappa(r/\rho)=\frac1\rho\arctan
 \frac{\rho R}{2\sqrt{1-\rho^2R^2/2}}
 =\frac R2+\frac{\rho^2R^3}{12}+O(\rho^4).
\end{equation}
The strictly decreasing function
\(h_*(R)=\kappa(R)-R/2\) has a unique zero \(R_s\)
within the interval, and \(h_*'(R_s)=\kappa'(R_s)-1/2<0\).
Equation \eqref{eq:root-expansion} and the mean value theorem
therefore give the quantitative estimates
\begin{equation}\label{eq:root-rate}
 R_\rho=R_s+O(\rho^2),\qquad
 r_\rho=\rho R_s+O(\rho^3).
\end{equation}
In particular this proves \eqref{eq:Rs-limit}.
\end{proof}

\section{Exact potentials and the scaled functions}

The rescaling viewpoint used in this section is related to the
cotangent-coordinate limiting analysis in \cite[\S2.2]{Jiao};
the source and target potentials below are specific to the present
Lagrangian-intersection construction and are derived explicitly.

For \(\sigma=\pm1\), use the canonical chart
\eqref{eq:nz} near \(P_\sigma\).

\begin{lemma}\label{lem:target}
Near \(P_\sigma\), \(L_*\) is exactly the local Lagrangian
graph over the momentum coordinates determined by
\begin{equation}\label{eq:target-graph}
 q=\nabla F_\sigma(p),
\end{equation}
where
\begin{equation}\label{eq:F}
 F_\sigma(p)=-\sigma I(p_1p_2),\qquad
 I(w)=\int_0^w
 \frac{\dd z}{\sqrt{2(1+\sqrt{1+z^2})}}.
\end{equation}
In particular,
\(F_\sigma(p)=-\sigma p_1p_2/2+O(|p|^6)\).
\end{lemma}
\begin{proof}
Since
\[
L_*=\{j_y=0,\ k_x=0\},
\]
we rewrite these two defining equations in the canonical coordinates
\((q,p)\) of Lemma~\ref{lem:coordinates}. Using \eqref{eq:S-p} and \eqref{eq:inverse}, they become
\[
 q_2=-\frac{n_z}{2a}p_1,\qquad
 q_1=-\frac{n_z}{2a}p_2.
\]
Write \(T=|n_z|/(2a)>0\), so
\(q=-\sigma T(p_2,p_1)\).
Substitution into \(n_z^2=1-|q|^2\) and
\(a^2=1-\bigl(|p|^2-(q\cdot p)^2\bigr)/4\) gives
\[
 4T^2+4T^4(p_1p_2)^2=1.
\]
The positive solution is
\[
 T=\frac1{\sqrt{2(1+\sqrt{1+(p_1p_2)^2})}}.
\]
Thus \(q=\grad F_\sigma(p)\). Conversely, this expression satisfies
the two target equations for all sufficiently small \(p\).
Their derivative in \(q\) is invertible at \(q=p=0\), so the
implicit function theorem gives uniqueness of this local branch.
Since \(I'(w)=1/2+O(w^2)\), the expansion follows.
\end{proof}

Let \(\mathcal A=\{v\in\R^2:1<|v|<2\}\).  We write
\(E\Subset\mathcal A\) to mean that \(\overline E\) is a compact subset of
\(\mathcal A\).  In statements on compact subsets of \(\mathcal A\), all
derivative bounds are taken after fixing the subset.

\begin{lemma}\label{lem:source}
On every compact momentum subannulus of \(\mathcal A\), and for
small \(\rho>0\), the part of \(L_\rho\) near either \(P_\sigma\)
is the local Lagrangian graph \(q=\grad A_\rho(p)\) over the momentum coordinates, with the same radial
potential in both charts.
For \(s=|p|\), this potential is specified by
\begin{equation}\label{eq:A-implicit}
 A_\rho'(s)=-\sin\bigl(\rho\kappa(\xi/\rho)\bigr),\qquad
 s=\xi/\cos\bigl(\rho\kappa(\xi/\rho)\bigr),\qquad
 A_\rho(2\rho)=0.
\end{equation}
More precisely, for any fixed \(C>0\) and compact
\(E\Subset\mathcal A\), every point of \(L_\rho\) in the chart
with \(p=\rho v\), \(v\in E\), and \(|q|\leq C\rho\)
belongs to this graph.
The scaled potential
\begin{equation}\label{eq:a}
 a_\rho(v)=\rho^{-2}A_\rho(\rho|v|)
\end{equation}
converges in \(C^\infty\) on such compact sets to
\[
 a_0(v)=\int_{|v|}^2\kappa(R)\dd R.
\]
\end{lemma}
\begin{proof}
In the coordinates of Lemma~\ref{lem:coordinates}, the original
torus \(L\), for \(0<r<2\), satisfies
\[
 n_z=0,\qquad \eta=\pm rN_0.
\]
Indeed \(j_z=k_z=0\) is equivalent to \(n_z=S_z=0\);
then both \(S\) and \(n\) are horizontal and satisfy \(S\perp n\);
hence \(\eta=S\times n\) is vertical. 
Since the Hamiltonian equations for \(r=|S|\) are
\[
\dot j=\frac Sr\times j,\qquad
\dot k=\frac Sr\times k,
\]
we have
\[
\dot S=\frac Sr\times S=0.
\]
Hence \(S\) (and therefore \(r\)) is constant along the flow. The identity
\[
|j-k|^2+|S|^2=4
\]
then implies that \(|j-k|\) is also constant. Consequently, using
\(n=(j-k)/|j-k|\) and \(\eta=S\times n\), we obtain
\[
\dot n
=\frac{\dot j-\dot k}{|j-k|}
=\frac Sr\times\frac{j-k}{|j-k|}
=\frac Sr\times n
=\frac{\eta}{r}.
\]
Using \(S\perp n\),
\[
\dot\eta
=S\times\dot n
=\frac1r S\times(S\times n)
=-rn.
\]
Its solution through \((n_0,\eta_0)\) is
\begin{equation}\label{eq:geodesic}
 n(t)=\cos t\,n_0+\sin t\,\eta_0/r,\qquad
 \eta(t)=\cos t\,\eta_0-r\sin t\,n_0.
\end{equation}
Under the cotangent-bundle identification of Lemma~\ref{lem:coordinates}, let
\(\widetilde\phi_r^t\) denote the Hamiltonian flow of \(r=|\eta|\). Let \(N_0=(0,0,1)\) as before. 
The two branches of \(L\) are
\[
\mathcal B_\pm
=
\{(n_0,\pm rN_0):(n_0)_z=0,\ |n_0|=1,\ 0<r<2\}.
\]
At time \(t=\pi/2\),
\[
\widetilde\phi_r^{\pi/2}(n_0,\pm rN_0)
=
(\pm N_0,-rn_0).
\]
Hence, as \(n_0\) and \(r\) vary, the two branches are carried to the
cotangent fibers over \(N_0\) and \(-N_0\), respectively, inside
\(\{|\eta|<2\}\).

For \(r>\rho\), the angle in \(\Phi_\rho\) is
\(\pi/2-\delta(r)\), where \(\delta(r)=\rho\kappa(r/\rho)\).
Since \(r\) is preserved by its Hamiltonian flow, we have
\[
\widetilde\phi_r^{\pi/2-\delta(r)}
=
\widetilde\phi_r^{-\delta(r)}
\circ\widetilde\phi_r^{\pi/2}.
\]
Thus, rather than flowing from the original branch
\(\mathcal B_\sigma\) for time \(\pi/2-\delta(r)\), we may
first identify its time-\(\pi/2\) image with the cotangent
fiber over \(\sigma N_0\), and then flow backward for time
\(\delta(r)\). 
Starting at a point \((\sigma N_0,w)\) of this fiber,
where \(r=|w|\), and applying \eqref{eq:geodesic}
with \(t=-\delta(r)\), we obtain
\[
 q=-\sin\delta\,w/r,\quad n_z=\sigma\cos\delta,\quad
 \eta_h=\cos\delta\,w,\quad \eta_z=\sigma r\sin\delta.
\]
Here \(\eta_h\) denotes the two horizontal components.
The definition of canonical \(p\) (cf. \eqref{eq:p-from-eta}) yields
\begin{equation}\label{eq:source-graph}
 p=w/\cos\delta,\qquad
 |p|=r/\cos\delta,\qquad q=-\sin\delta\,p/|p|.
\end{equation}
This is exactly the radial gradient in \eqref{eq:A-implicit}.
On any fixed compact range of \(r/\rho>1\), the derivative of
\(r/\cos\delta(r)\) is
\[
 \sec\delta\bigl(1+r\tan\delta\,\kappa'(r/\rho)\bigr)=1+O(\rho^2),
\]
uniformly with all derivatives on each fixed compact subannulus
of \(1<r/\rho<2\). Hence this derivative is at least \(1/2\)
for sufficiently small \(\rho\), and the relation
\(
|p|=r/\cos\delta(r)
\)
uniquely determines $r$ as a smooth function of $|p|$
on the radial interval under consideration.

For the assertion about all branches in a bounded scaled chart,
\eqref{eq:canonical} gives
\[
 \frac r\rho=
 \sqrt{|v|^2-(q\cdot v)^2}.
\]
If \(v\in E\) and \(|q|\leq C\rho\), this quantity is greater
than \(1\) for small \(\rho\). Thus the angle is precisely
\(\pi/2-\rho\kappa(r/\rho)\), and the preceding parametrization
of \(L_\rho\) applies. The sign of $n_z$ determines the branch
$\sigma\in\{+1,-1\}$.
On each fixed branch, the strict monotonicity of
\[
G_\rho(r)=\frac{r}{\cos\delta(r)}
\]
implies that the equation
\[
|p|=G_\rho(r)
\]
uniquely determines
$r=G_\rho^{-1}(|p|)$.
Equation \eqref{eq:source-graph} then uniquely
determines
\[
q=-\sin\delta\bigl(G_\rho^{-1}(|p|)\bigr)
   \frac{p}{|p|}.
\]
Thus, for each fixed $\sigma$, the projection
of this branch of $L_\rho$ onto the momentum
coordinates is injective.
In particular, every point of $L_\rho$ in the
bounded scaled chart under consideration
belongs to the single momentum graph
$q=\nabla A_\rho(p)$, and no additional
local branches occur.

To prove convergence, recall that
\[
G_\rho(\xi)
=\frac{\xi}{\cos\bigl(\rho\kappa(\xi/\rho)\bigr)}
\]
is strictly increasing and hence admits a smooth
inverse on its image. For the scaled radial variable
$s=\rho R$, define
\[
\Xi_\rho(R)
:=\frac{1}{\rho}G_\rho^{-1}(\rho R).
\]
Thus the original radial variable is uniquely given by
$\xi=\rho\Xi_\rho(R)$. By the definition of $G_\rho$,
\[
G_\rho\bigl(\rho\Xi_\rho(R)\bigr)=\rho R,
\]
which is equivalent to
\[
\Xi_\rho(R)
=R\cos\bigl(\rho\kappa(\Xi_\rho(R))\bigr).
\]

Fix a compact interval $I\Subset(1,\infty)$ of the
scaled radial variable $R$. Since
$\kappa\in C^\infty(\mathbb R)$, define the smooth
function $F:\mathbb R^3\to\mathbb R$ by
\[
F(\Xi,R,\rho)
:=\Xi-R\cos\bigl(\rho\kappa(\Xi)\bigr).
\]
The preceding implicit equation is then
\[
F(\Xi_\rho(R),R,\rho)=0.
\]
At $\rho=0$, we have
\[
F(R,R,0)=0.
\]
Moreover,
\[
\frac{\partial F}{\partial\Xi}
=1+\rho R\sin\bigl(\rho\kappa(\Xi)\bigr)
  \kappa'(\Xi),
\]
so
\[
\frac{\partial F}{\partial\Xi}(R,R,0)=1\ne0.
\]
The implicit function theorem, applied uniformly
for $R\in I$, yields a unique smooth solution
$\Xi=\Xi(R,\rho)$ for sufficiently small
$|\rho|$, with $\Xi(R,0)=R$.
Consequently,
\[
\Xi_\rho\longrightarrow\operatorname{id}
\quad\text{in }C^\infty(I)
\]
as $\rho\to0$.  The function $F(\Xi,R,\rho)$ is even in \(\rho\), and differentiation
at \(\rho=0\) (or Taylor's formula) gives, on any fixed compact
radial interval, for each nonnegative integer \(k\),
\[
 \|\Xi_\rho-\operatorname{id}\|_{C^k}
 \le C_k\rho^2.
\]
To distinguish the two-variable radial potential
$a_\rho(v)$ from its one-variable radial profile,
define
\[
\widehat a_\rho(R)
:=\rho^{-2}A_\rho(\rho R),
\qquad
a_\rho(v)=\widehat a_\rho(|v|).
\]
By \eqref{eq:A-implicit} and the definition of
$\Xi_\rho$, we have
\[
\frac{d\widehat a_\rho(R)}{dR}
=-\frac{\sin(\rho\kappa(\Xi_\rho(R)))}{\rho}.
\]
Since $\widehat a_\rho(2)=0$, integration yields
\[
\widehat a_\rho(R)
=\int_R^2
\frac{\sin(\rho\kappa(\Xi_\rho(u)))}{\rho}
\dd u.
\]
Similarly, define
\[
\widehat a_0(R)=\int_R^2\kappa(u)\dd u,
\qquad
a_0(v)=\widehat a_0(|v|).
\]
Since $\Xi_\rho=\operatorname{id}+O(\rho^2)$
in $C^\infty$ on fixed compact radial intervals,
and $\sin(\rho z)/\rho=z+O(\rho^2)$ smoothly
with all derivatives on bounded sets, we obtain
\[
\|\widehat a_\rho-\widehat a_0\|_{C^k(I)}
\le C_{I,k}\rho^2
\]
for every fixed compact radial interval
$I\subset(1,2]$.
As $v\mapsto |v|$ is smooth away from the
origin, it follows that
\[
\|a_\rho-a_0\|_{C^k(E)}
\le C_{E,k}\rho^2
\]
for every fixed compact $E\Subset\mathcal A$,
where the norms are taken in the
two-dimensional $v$-variables.
This proves the stated smooth convergence
and records a uniform rate needed below.
\end{proof}

In the following we use a rescaling technique of the difference of the two potential functions to preserve critical-point
nondegeneracy and Morse index. The reason is that the rescaling freezes the shrinking intersection region into a fixed annulus and produces a nontrivial smooth limiting model, allowing the saddle-merging construction and its support estimates to be carried out uniformly before transferring the exact modification back to the original symplectic coordinates.

Define the scalar functions
\begin{equation}\label{eq:f-rho}
 f_{\rho,\sigma}(v)=\rho^{-2}F_\sigma(\rho v)-a_\rho(v).
\end{equation}
Their critical points describe the intersections of the two local
Lagrangian graphs over the momentum coordinates.  Indeed, for
\(D_{\rho,\sigma}(p)=F_\sigma(p)-A_\rho(p)\) we have
\(f_{\rho,\sigma}(v)=\rho^{-2}D_{\rho,\sigma}(\rho v)\), whence
\(\nabla_v f_{\rho,\sigma}=\rho^{-1}\nabla_pD_{\rho,\sigma}\)
and
\(D_v^2f_{\rho,\sigma}=D_p^2D_{\rho,\sigma}\)
at corresponding points.  
A critical point is nondegenerate if its Hessian is
invertible; in two variables a saddle is a nondegenerate critical
point with one positive and one negative Hessian eigenvalue.
These are the conventions in \cite[\S2]{Milnor}.

\begin{proposition}\label{prop:limit}
On compact subsets of \(\mathcal A\),
\begin{equation}\label{eq:f-zero}
 f_{\rho,\sigma}\longrightarrow
 f_{0,\sigma}(v)=h(|v|)-\sigma v_1v_2/2
 \quad\text{in }C^\infty,\qquad
 h(R)=-\int_R^2\kappa(u)\dd u.
\end{equation}
More precisely, for each compact \(E\Subset\mathcal A\)
and integer \(k\ge0\), a constant \(C_{E,k}>0\), independent
of sufficiently small \(\rho\), satisfies
\begin{equation}\label{eq:f-rate}
 \|f_{\rho,\sigma}-f_{0,\sigma}\|_{C^k(E)}
 \le C_{E,k}\rho^2,\qquad \sigma=\pm1.
\end{equation}
Each \(f_{0,\sigma}\) has exactly two critical points in
\(\mathcal A\), both saddles and both at radius \(R_s\).
The two actual intersections near \(P_\sigma\) converge, in
scaled momentum, to these two critical points.
All four intersections in Proposition~\ref{prop:four} are
transverse for small \(\rho\).
\end{proposition}
\begin{proof}
Lemma~\ref{lem:source} gives
\(\|a_\rho-a_0\|_{C^k(E)}=O(\rho^2)\). Since the function $I$ in \eqref{eq:F} satisfies 
\(I(w)=w/2+O(w^3)\) smoothly near \(w=0\), the target term
satisfies, with all \(v\)-derivatives on \(E\),
\[
 \rho^{-2}F_\sigma(\rho v)
 =-\frac{\sigma}{2}v_1v_2+O(\rho^4).
\]
This, together with Lemma~\ref{lem:source},  proves \eqref{eq:f-rate} and
\eqref{eq:f-zero}. For each \(\sigma\in\{\pm1\}\), introduce the orthogonal coordinates
\[
x=\frac{v_1-\sigma v_2}{\sqrt2},
\qquad
y=\frac{v_1+\sigma v_2}{\sqrt2}.
\]
Then \(R=|v|=\sqrt{x^2+y^2}\) and
\[
-\frac{\sigma}{2}v_1v_2=\frac{x^2-y^2}{4}.
\]
Hence the limit function becomes
\begin{equation}\label{eq:planar-model}
g(x,y)=h(R)+\frac{x^2-y^2}{4}.
\end{equation}
Its critical equations are
\[
 \bigl(\kappa(R)/R+1/2\bigr)x=0,\qquad
 \bigl(\kappa(R)/R-1/2\bigr)y=0.
\]
Since \(\kappa(R)\ge0\), the factor
\(
\kappa(R)/R+1/2
\)
is strictly positive, so the first critical equation gives \(x=0\).
Since \(R>1\), we cannot also have \(y=0\). Hence the second equation
implies
\[
\kappa(R)=\frac R2.
\]
By the uniqueness of \(R_s\), we obtain \(R=R_s\), and therefore
\(y=\pm R_s\). Thus the only critical points are
\[
(x,y)=(0,\pm R_s).
\]
A direct computation shows that the  eigenvalues of the Hessian of the function $g(x,y)$ are \(1\) and
\(\kappa'(R_s)-1/2<0\).

By Proposition~\ref{prop:four}, every actual intersection has
\(r=|j+k|=r_\rho\), where \(r_\rho\) is the unique solution
of \eqref{eq:root}. By \eqref{eq:Rs-limit},
\(r_\rho/\rho\to R_s\). Hence \eqref{eq:source-graph} gives
\[
 \frac{|p|}{\rho}
 =\frac{r_\rho/\rho}{\cos\delta(r_\rho)}
 \longrightarrow R_s,
\]
since \(\delta(r_\rho)=\rho\kappa(r_\rho/\rho)\to0\).

We next determine the limiting directions. 
Equating the local graph equations for \(L_\rho\)
and \(L_*\), given by \eqref{eq:source-graph}
and \eqref{eq:target-graph}, respectively, and using
\eqref{eq:F}, we obtain

\[
-\sin\delta(r_\rho)\frac{p}{|p|}
=\nabla F_\sigma(p)
=-\sigma I'(p_1p_2)(p_2,p_1).
\]
Since \(p\neq0\), \(\sigma\in\{\pm1\}\),
\(\sin\delta(r_\rho)>0\), and \(I'(p_1p_2)>0\),
set
\[
a:=\frac{\sin\delta(r_\rho)}{|p|}>0,
\qquad
b:=I'(p_1p_2)>0.
\]
The component equations are
\[
ap_1=\sigma bp_2,
\qquad
ap_2=\sigma bp_1.
\]
Since \(p\neq0\), these equations imply
\(a^2=b^2\). 
As \(a,b>0\), we obtain \(a=b\), and hence
$p_2=\sigma p_1$.

Now set \(v=p/\rho\) and use the orthogonal coordinates
\[
 x=\frac{v_1-\sigma v_2}{\sqrt2},
 \qquad
 y=\frac{v_1+\sigma v_2}{\sqrt2}.
\]
The relation \(p_2=\sigma p_1\) gives \(x=0\), while
\(|v|\to R_s\) gives \(y\to\pm R_s\). 
Thus the two actual intersections near \(P_\sigma\)
converge, in these coordinates, to the two critical
points \((0,\pm R_s)\) of \eqref{eq:planar-model}.
At both critical points, the \(y\)-direction corresponds
to the negative eigenvalue of the Hessian.
 Their Hessians converge to the corresponding
invertible limiting Hessians by the \(C^2\) convergence
already established. Lemma~\ref{lem:graph} therefore proves
that all four intersections are transverse for sufficiently
small \(\rho\).

\end{proof}

\section{A relative modification of the planar functions}

For a smooth function \(f\), \(\Crit(f)\) denotes its critical set.  An
embedded regular level arc between two saddles means a smooth embedded arc in
their common level set whose only critical points are its two endpoints.  The
term is used here solely in this sense.
The next two lemmas provide its existence and the needed relative
modification.  The underlying local degeneration is the classical elliptic
umbilic recalled in the introduction \cite{Thom72,Arnold72Sing}; the first
lemma verifies that the two saddles of our specific limiting function satisfy
the required same-level arc condition.  Its level-set argument is closely
related to the recent implementation in \cite[\S3.1]{Jiao}, but we include
the details because the coefficients and the radial profile in the present
Lagrangian-intersection problem are different.

\begin{lemma}\label{lem:arc}
The two saddles of \(g\) in \eqref{eq:planar-model} are joined
by an embedded regular level arc contained in \(1<R<2\).
The arc extends smoothly and a short distance beyond both saddles,
still in the same level set.
It has a disk neighborhood compactly contained in this annulus
and containing no other critical points.
\end{lemma}
\begin{proof}

Introducing polar coordinates
\[
x=R\cos\vartheta,\qquad
y=R\sin\vartheta,
\]
we can write
\[
g(R,\vartheta)
=h(R)+\frac{R^2}{4}\cos(2\vartheta),
\qquad
c_s=h(R_s)-\frac{R_s^2}{4}.
\]
Here $R_s$ is given in \eqref{eq:Rs-limit}. 
Since \(\kappa=100\) on \([1,5/4]\) and
\(\kappa\geq0\), we have
\[
h(1)
=-\int_1^2\kappa(u)\,du
\leq-\int_1^{5/4}100\,du
=-25.
\]
Consequently,
\[
g(1,\vartheta)
=h(1)+\frac14\cos(2\vartheta)
\leq-25+\frac14.
\]
On the other hand, since \(\kappa\) is nonincreasing
and \(\kappa(R_s)=R_s/2\), we obtain
\[
\begin{aligned}
c_s
&=-\int_{R_s}^2\kappa(u)\,du
  -\frac{R_s^2}{4}\\
&\geq-\frac{R_s}{2}(2-R_s)
  -\frac{R_s^2}{4}\\
&=-1+\frac{(2-R_s)^2}{4}
>-1.
\end{aligned}
\]
Thus, for every \(\vartheta\),
\[
g(1,\vartheta)
\leq-25+\frac14
<-1<c_s.
\]
For \(1<R<R_s\),
\[
 \partial_Rg\geq\kappa(R)-R/2>0.
\]
At \(R=R_s\), \(g\geq c_s\), with equality exactly in
the two \(y\)-axis directions. Hence in each direction
\(-\pi/2<\vartheta<\pi/2\) there is a unique radius
in \((1,R_s)\) where \(g=c_s\). 
By the implicit function theorem, the unique radial
solutions \(R=R(\vartheta)\) depend smoothly on
\(\vartheta\). Consequently, the curve
\[
\gamma(\vartheta)=
\bigl(R(\vartheta)\cos\vartheta,\,
R(\vartheta)\sin\vartheta\bigr),
\qquad
-\frac{\pi}{2}<\vartheta<\frac{\pi}{2},
\]
is a smooth embedded arc in the open right half-plane.
They tend to the two saddles at the endpoints.

\begin{figure}[htbp]
    \centering
    \includegraphics[width=0.55\textwidth]{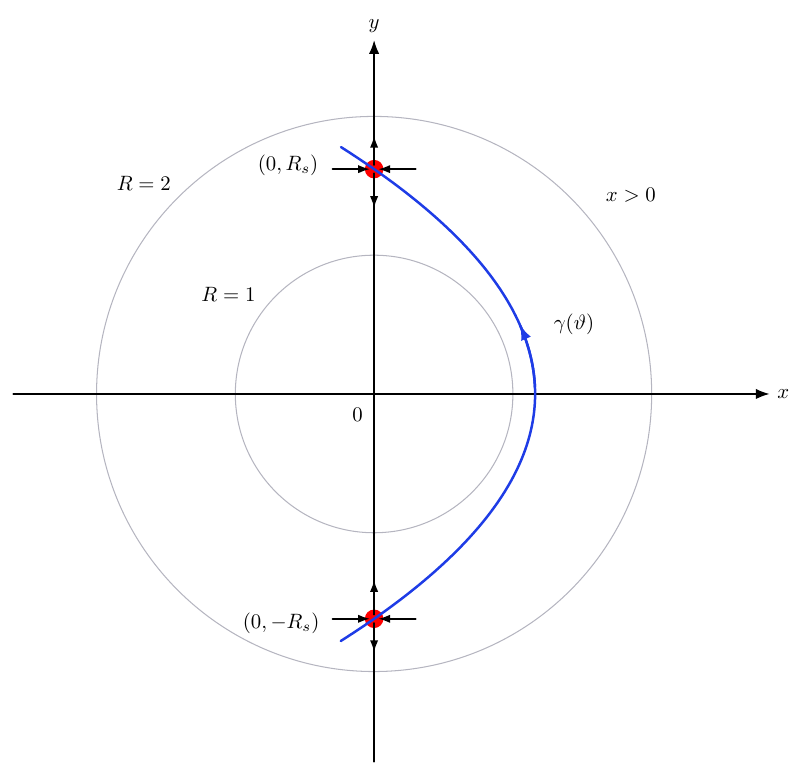}
    \caption{An embedded arc in the annular region $\mathcal A$ passing through the two saddles.}
    \label{fig:lemma51_arc}
\end{figure}

At each saddle, the Morse lemma
\cite[Lemma~2.2]{Milnor} expresses the level set as two transverse
smooth branches. The chosen arc approaches one branch and therefore
continues smoothly through that endpoint.
Choose short continuations in the annulus, and a sufficiently narrow
tubular neighborhood of the resulting compact embedded arc (see~Figure~\ref{fig:lemma51_arc}).
It is a disk, compactly contained in the annulus.
Proposition~\ref{prop:limit} excludes any further critical point.
\end{proof}

The fusion of critical points on a connected critical level has a
classical antecedent in Takens' work
\cite[Theorem~(2.7) and the proof of Proposition~(2.9)]{Takens68}.
Lemma~\ref{lem:merge} supplies the more specific local conclusion needed here:
a compactly supported replacement of two Morse saddles by one critical point
with the prescribed \(D_4^-\) cubic normal form
\cite{Thom72,Arnold72Sing}.  Takens' general fusion argument does not
assert this normal form.  A related recent construction appears in
\cite[Lemma~3.1]{Jiao}.  We include a proof that tracks the coordinate
normalization, support control, and exclusion of additional critical points.

\begin{lemma}[Relative combination of two saddles]\label{lem:merge}
Let \(f\) be smooth on a neighborhood of an embedded regular level
arc between two nondegenerate saddles of the same value \(c\).
Assume the arc has short smooth level-set extensions beyond both
endpoints. In a sufficiently narrow disk neighborhood \(D\) of the
extended arc, there is a smooth \(\widetilde f\) with
\[
 \supp(\widetilde f-f)\Subset D,\qquad
 \#\Crit(\widetilde f|_D)=1.
\]
At its only critical point, a smooth coordinate change expresses
\(\widetilde f\) as \(c+t(t^2-s^2)\).
\end{lemma}
\begin{proof}
\medskip
\noindent\textit{Step 1: Choose coordinates adapted to the common level arc and put the function into a preliminary normal form.}
We first arrange that the connecting level arc is the coordinate axis and that the two saddles occur at fixed parameter values.
Subtract the common critical value \(c\) from the function $f$, so that the connecting arc lies in the zero level set. Before fixing the longitudinal coordinate \(s\) along the extended arc, choose an even smooth cutoff profile \(\beta_0:\mathbb R\to[0,1]\), equal to one on \([-1/2,1/2]\), zero outside \((-1,1)\), and nonincreasing on $[0,\infty)$. Define
\[
C_\beta=\sup_{u\in\mathbb R}
|\beta_0(u)+u\beta_0'(u)|<\infty,
\]
and fix \(Q>1\) so that \(Q^2>4+C_\beta\).
As the regular level arc continues for a
short positive distance beyond each saddle, we may reparameterize this
\emph{fixed geometric extended arc} by an interval containing
\([-Q-2,Q+2]\), taking the saddles to \(s=-1\) and \(s=1\).
This change only stretches the parameter along the same embedded arc;
it does not require any further geometric continuation of the level set.
Choose a transverse coordinate \(z\) so that the arc is \(z=0\).
The coordinate can be obtained by taking a smooth unit normal to
the embedded arc and using \(\gamma(s)+z\nu(s)\);
compactness and the inverse function theorem give an embedded
chart when its width is small.
The identity \(f(s,0)=0\) gives
\[
 f(s,z)=z A(s)+z^2b_0(s,z),\qquad A(s)=f_z(s,0).
\]
For example, the first division follows by integrating \(f_z(s,uz)\)
over \(0\leq u\leq1\), and a second such integration gives the
quadratic remainder.
Regularity away from the two saddles, including their short
extensions, gives \(A(s)\ne0\) there.  Reverse \(z\) if necessary
so that \(A(s)>0\) between the saddles.  At either saddle,
\[
 D^2f(s,0)=
 \begin{pmatrix}0&A'(s)\\A'(s)&2b_0(s,0)\end{pmatrix}.
\]
Nondegeneracy implies \(A'(\pm1)\neq0\).
The simple zeros \(A(\pm1)=0\) and the sign of \(A\) on the
middle interval imply that \(A'(-1)>0\) and \(A'(1)<0\).
Since \(A\) has no other zeros on the chosen extended arc,
\(c_0(s)=A(s)/(1-s^2)\) extends smoothly and stays strictly
positive throughout the whole chosen parameter interval (also for
\(|s|>1\), where both numerator and denominator are negative).
With \(t=c_0(s)z\), the expression becomes
\[
 f(s,t)=t(1-s^2)+t^2b(s,t).
\]

\medskip
\noindent\textit{Step 2: Remove the remaining quadratic error by a smooth change of coordinates and obtain the standard two-saddle template.}
We use Moser's path method, constructing a smooth homotopy of functions and a time-dependent vector field whose flow realizes the required change of coordinates.
We now show that the preceding expression is smoothly equivalent, on a sufficiently narrow tube, to \(t(t^2+1-s^2)\).
For \(0\leq\lambda\leq1\), put
\[
 f_\lambda=t(1-s^2)+(1-\lambda)t^2b(s,t)+\lambda t^3.
\]
Set \(B=t-b(s,t)\), so \(\partial_\lambda f_\lambda=t^2B\).
On the regions \(|s|\ge 1/2\), write
\[
\partial_s f_\lambda=tA_\lambda,
\qquad
A_\lambda=-2s+(1-\lambda)t b_s.
\]
Since \(b_s\) is bounded on a fixed compact coordinate neighborhood, say \(|b_s|\le M\), we have
\[
|A_\lambda|
\ge 2|s|-|t|M
\ge 1-|t|M.
\]
Thus, by choosing the tube sufficiently narrow, we obtain
\(
|A_\lambda|\ge\frac12
\)
uniformly for \(0\le\lambda\le1\). There the smooth vector field
\[
 X_\lambda=(-tB/A_\lambda,0)
\]
solves \(\mathrm df_\lambda(X_\lambda)=-t^2B\).
In a middle region \(|s|<3/4\), \(\partial_tf_\lambda\) is
bounded away from zero on a narrow tube, and the vector field
\[
 X_\lambda=(0,-t^2B/\partial_tf_\lambda)
\]
is smooth. Choose the tube so that the first formula is valid on
\(|s|\geq1/2\) and the second on \(|s|\leq3/4\). On the overlap both
vector fields solve the same scalar equation
\(\mathrm df_\lambda(X)=-t^2B\). A partition of unity in \(s\),
subordinate to these two regions, therefore glues them to a smooth
\(X_\lambda\) satisfying the equation everywhere; see Figure~\ref{fig:moser-vector-field}. Since each local
solution vanishes on \(t=0\), so does the glued field.

\begin{figure}[htbp]
    \centering
    \includegraphics[width=0.58\textwidth]{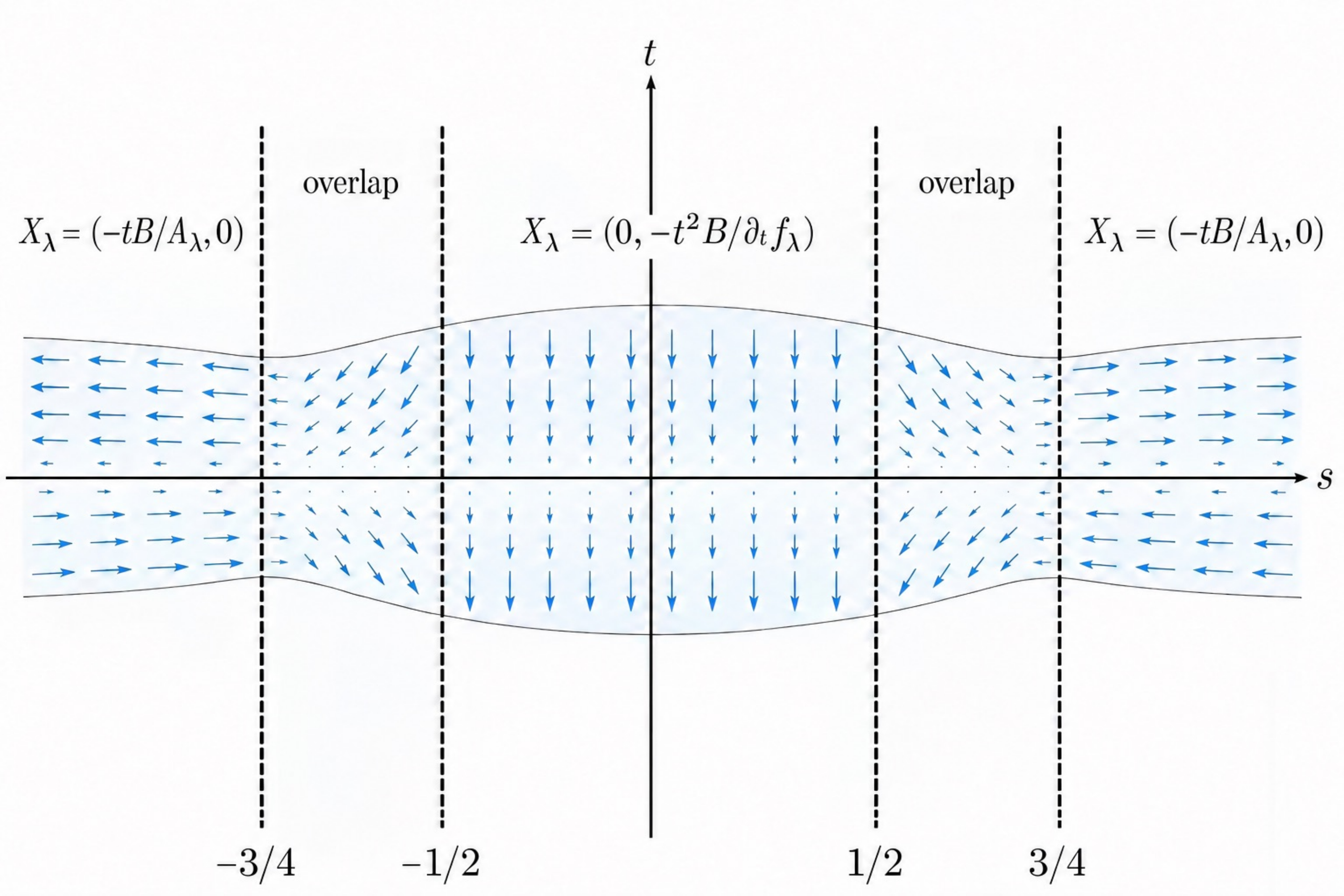}
    \caption{A schematic illustration of the Moser
    vector field $X_\lambda$.}
    \label{fig:moser-vector-field}
\end{figure}

Fix first a compact \(s\)-interval strictly inside the extended
arc, together with a slightly larger one on which the vector fields
above are defined. After these longitudinal intervals are fixed, choose a tube of positive width 
on which the absolute values of the denominators $A_\lambda$ and $\partial_tf_\lambda$ are 
uniformly bounded away from zero for all $0\leq \lambda\leq 1$.
The resulting constant \(C_0\) in
\(|X_\lambda(s,t)|\le C_0|t|\) depends on these fixed
coordinate data, not on the subsequently chosen initial tube width. 
The differential inequality for the transverse coordinate, together with Gronwall's inequality, gives
\(|t(\lambda)|\le e^{C_0}|t(0)|\); the displacement in
\(s\) is at most \(C_0e^{C_0}|t(0)|\) due to $|s'(\tau)|\leq C_0|t(\tau)|$.
Choosing the initial width sufficiently small now ensures that
trajectories stay in the larger tube for \(0\le\lambda\le1\).
If \(\phi_\lambda\) denotes the flow time-dependent vector field $X_\lambda$, then
\[
 \frac{\mathrm d}{\mathrm d\lambda}
  (f_\lambda\circ\phi_\lambda)
 = (\partial_\lambda f_\lambda
    +\mathrm df_\lambda(X_\lambda))\circ\phi_\lambda=0,
 \qquad f_1\circ\phi_1=f_0.
\]
Thus, using \(\phi_1\) as the new coordinate map, the original
function has the standard form
\begin{equation}\label{eq:template}
 t(t^2+1-s^2).
\end{equation}

\medskip
\noindent\textit{Step 3: Perform a compactly supported deformation of the template that merges the two saddles into one cubic critical point, while creating no additional critical points.}
We use a compactly supported version of the classical \(D_4^-\) unfolding from Thom--Arnol'd singularity theory \cite{Thom72,Arnold72Sing}, in the spirit of Takens' critical-point fusion construction \cite{Takens68}. The cutoff parameters are chosen to prevent the appearance of additional critical points. 
We deform \eqref{eq:template} only inside a compact coordinate rectangle and keep it unchanged near the boundary.
Use the profile \(\beta_0\), the bound \(C_\beta\), and the
number \(Q\), all fixed before the coordinate normalization in
Step~1.  Now choose \(0<\eps<1\) so small that
\(3\eps^2+1-Q^2+C_\beta<0\) and the closed rectangle \(|s|\leq Q+1,\ |t|\leq\eps\) is
contained in the normalized chart. Set \(\beta(t)=\beta_0(t/\eps)\),
and choose \(0\leq\alpha\leq1\), equal to one on \(|s|\leq Q\)
and zero on \(|s|\geq Q+1\).
For \(0\leq b\leq1\), define
\begin{equation}\label{eq:Fb}
 F_b(s,t)=t\bigl(t^2+1-s^2-b\alpha(s)\beta(t)\bigr).
\end{equation}
Thus \(F_0\) is exactly the template \eqref{eq:template}, and the deformation is supported in the interior of the coordinate rectangle.
The two critical-point equations are
\[
 \begin{aligned}
  \partial_sF_b&=-2st-bt\alpha'(s)\beta(t)=0,\\
  \partial_tF_b&=3t^2+1-s^2
           -b\alpha(s)\bigl(\beta(t)+t\beta'(t)\bigr)=0.
 \end{aligned}
\]
In the core \(|s|\leq Q\) the first equation is
\(\partial_sF_b=-2st\)=0.  If \(t=0\), the second reduces
to \(s^2=1-b\), giving precisely the two saddles for
\(b<1\) (with \(s=\pm1\) when \(b=0\)) and the origin
when \(b=1\).  If \(t\ne0\), the first equation forces
\(s=0\).  Since \(t\beta'(t)\le0\), we then have
\[
 \partial_tF_b(0,t)
 =3t^2+1-b(\beta(t)+t\beta'(t))
 \ge3t^2+1-b\beta(t)>0\qquad(t\ne0),
\]
so there is no such critical point.  In the remaining horizontal
transition region \(|s|>Q\), \(|t|\le\eps\), one has
\[
 \partial_tF_b
 \le 3\eps^2+1-Q^2+C_\beta<0,
\]
regardless of the derivative of \(\alpha\).
Thus no critical point occurs there. For \(|t|\geq\eps\), the cutoff
\(\beta\) vanishes, so the template is unchanged; its only critical points
are \((\pm1,0)\), which lie in the already treated region
\(|s|\leq Q\). Hence \(F_1\) has exactly one critical point, at the
origin, with local expression \(t(t^2-s^2)\).

\medskip
\noindent\textit{Step 4: Transport the local deformation back to the original disk and verify the relative support condition.}
Finally, pull \(F_1\) back through the coordinate changes constructed above, restore the critical value \(c\), and leave \(f\) unchanged outside the modification region.
Transporting \(F_1\) back to \(D\), adding \(c\), and extending
by \(f\) gives \(\widetilde f\).
Because the deformation from \(F_0\) to \(F_1\) is supported in the
interior of the normalized rectangle, this construction gives
\(\supp(\widetilde f-f)\Subset D\).
All choices of profiles and the required \(s\)-interval preceded
the final choice of \(\eps\), so the width choices are consistent.
The preceding critical-point analysis shows that
\(\#\Crit(\widetilde f|_D)=1\), and at that point the local form is
\(c+t(t^2-s^2)\), as required.
\end{proof}

\section{Exact correction and Hamiltonian realization}

The exact-correction step below is adapted from the cutoff interpolation
used in \cite[\S4]{Jiao}.  Here it is applied directly to the difference
of the two Lagrangian graph potentials.

For each \(\sigma\in\{\pm1\}\), the orthogonal change of variables in Proposition~\ref{prop:limit} transforms \(f_{0,\sigma}\) into the planar model \(g\) considered in Lemma~\ref{lem:arc}. We may therefore apply Lemmas~\ref{lem:arc} and \ref{lem:merge} to each \(f_{0,\sigma}\).
This gives disks and scalar functions with
\begin{equation}\label{eq:disks}
 \overline{D_\sigma}\subset D_{1,\sigma},\qquad
 \overline{D_{1,\sigma}}\subset\mathcal A,\qquad
 \supp(\widetilde f_{0,\sigma}-f_{0,\sigma})\Subset D_\sigma.
\end{equation}
Here the disks \(D_\sigma\) contain both limit saddles; their closures
can be chosen with smooth boundary. Each \(\widetilde f_{0,\sigma}\)
has exactly one critical point in \(D_\sigma\), of the cubic form
in Lemma~\ref{lem:merge}.
Take \(\chi_\sigma\in C_c^\infty(D_{1,\sigma},[0,1])\),
equal to one on a neighborhood of \(\overline{D_\sigma}\), and put
\begin{equation}\label{eq:hat}
 \widehat f_{\rho,\sigma}
 =f_{\rho,\sigma}
 +\chi_\sigma(f_{0,\sigma}-f_{\rho,\sigma})
 +(\widetilde f_{0,\sigma}-f_{0,\sigma}).
\end{equation}
The final parenthesis is extended by zero outside its support.

\begin{lemma}\label{lem:exact-correction}
For sufficiently small \(\rho>0\), \(\widehat f_{\rho,\sigma}\)
has exactly one critical point in \(D_{1,\sigma}\).
It equals \(\widetilde f_{0,\sigma}\) on \(D_\sigma\), and
\[
 d_{\rho,\sigma}:=
 f_{\rho,\sigma}-\widehat f_{\rho,\sigma}
 \in C_c^\infty(D_{1,\sigma}).
\]
The functions \(a_\rho\) and \(d_{\rho,\sigma}\), with every
derivative on the relevant compact sets, are bounded uniformly
as \(\rho\to0\).
\end{lemma}
\begin{proof}
On \(D_\sigma\), equality follows directly from \eqref{eq:hat}.
On the compact set
\(K_\sigma=\overline{D_{1,\sigma}}\setminus D_\sigma\),
the gradient of \(f_{0,\sigma}\) is bounded away from zero,
since its only two critical points are inside \(D_\sigma\). Set
\[
 m_\sigma=\min_{K_\sigma}|\nabla f_{0,\sigma}|>0.
\]
Here the final parenthesis in \eqref{eq:hat} vanishes, and
\[
 \widehat f_{\rho,\sigma}-f_{0,\sigma}
 =(1-\chi_\sigma)(f_{\rho,\sigma}-f_{0,\sigma}).
\]
By the product rule and \eqref{eq:f-rate}, for a fixed
\(C_{\chi,\sigma}>0\) independent of small \(\rho\),
\[
 \bigl\|\nabla(\widehat f_{\rho,\sigma}-f_{0,\sigma})
     \bigr\|_{C^0(K_\sigma)}
 \le C_{\chi,\sigma}\rho^2<m_\sigma/2.
\]
Consequently, \( |\nabla\widehat f_{\rho,\sigma}|\ge
m_\sigma/2\) on \(K_\sigma\), and no additional critical
points occur there. The support statement follows from
\eqref{eq:disks} and the support of \(\chi_\sigma\).
Uniform bounds follow from the smooth convergence in
Proposition~\ref{prop:limit} and Lemma~\ref{lem:source}.
\end{proof}

We now choose the remaining constants in a fixed order.
The planar disks, coordinate changes and scalar modifications in
Section~5 are all fixed first. The functions \(\chi_\sigma\) are
then fixed. The bounds in Lemma~\ref{lem:exact-correction} allow
a constant \(C>0\), independent of all sufficiently small \(\rho\),
such that
\begin{equation}\label{eq:C}
 \bigl|\grad a_\rho(v)+\tau\grad d_{\rho,\sigma}(v)\bigr|
 <C,\qquad
 v\in\overline{D_{1,\sigma}},\quad0\leq\tau\leq1.
\end{equation}
Choose \(\zeta\in C_c^\infty(\R^2,[0,1])\) equal to one on
\(\{|u|\leq C+1\}\) and supported in \(\{|u|<C+2\}\).
Only now reduce \(\rho\) further, so that all conclusions above
hold, the two scaled coordinate supports below are disjoint
and contained in their hemisphere charts, and
Lemma~\ref{lem:source} applies to them with the bound \(C+2\).

We now transfer the scalar correction back to the original momentum scale and realize it by a compactly supported Hamiltonian shear, meaning that a symplectic transformation of the form \((q,p)\mapsto(q+\nabla H(p),p)\), which fixes the momentum coordinate and translates the position coordinate by the gradient of a function of \(p\). 
In the chart near \(P_\sigma\), define
\begin{equation}\label{eq:G}
 G_{\rho,\sigma}(q,p)=
 \rho^2\zeta(q/\rho)d_{\rho,\sigma}(p/\rho).
\end{equation}
It is extended by zero to \(M\).
This is a smooth global function: its momentum support is
compactly contained in \(\rho D_{1,\sigma}\), its position
support is contained in \(|q|<(C+2)\rho\), and both supports
lie in the interior of the coordinate chart.
The factor \(\rho^2\) reverses the normalization of the graph potentials, while the cutoff in \(q/\rho\) localizes the Hamiltonian without affecting the trajectories of interest.
Put
\[
 G_\rho=G_{\rho,+}+G_{\rho,-},\qquad
 \Psi_\rho=\varphi_{G_\rho}^1.
\]
Unlike the mixed-generating-function deformation used in
\cite[\S5]{Jiao}, we realize the exact correction here directly by a
compactly supported Hamiltonian shear in the canonical Lagrangian-graph
coordinates.

\begin{proposition}\label{prop:realization}
The Lagrangians \(\Psi_\rho(L_\rho)\) and \(L_*\) intersect
at exactly two points, one in each chart.
Near the point in the \(\sigma\)-chart, their potential difference
is \(\rho^2\widehat f_{\rho,\sigma}(p/\rho)\).
Both tangent planes coincide at these points.
\end{proposition}
\begin{proof}
For \(p=\rho v\), the original source graph is
\[
 q=\grad_p A_\rho(p)=\rho\grad a_\rho(v).
\]
As long as \(\zeta=1\), Hamilton's equations for
\eqref{eq:G} are
\[
 \dot p=0,\qquad
 \dot q=\rho\grad d_{\rho,\sigma}(v).
\]
Their proposed solution is
\begin{equation}\label{eq:solu}
 p(\tau)=\rho v,\qquad
 q(\tau)=\rho\bigl(\grad a_\rho(v)+
                    \tau\grad d_{\rho,\sigma}(v)\bigr).
\end{equation}
By \eqref{eq:C}, it stays in the region where \(\zeta=1\),
so it is the actual solution for \(0\leq\tau\leq1\).
Thus the new potential for $L_\rho$ is exactly
\[
 A_\rho^{\mathrm{new}}(p)
 =A_\rho(p)+\rho^2d_{\rho,\sigma}(p/\rho),
\]
and its difference from the potential for $L_*$ satisfies
\begin{equation}\label{eq:difference}
 F_\sigma(p)-A_\rho^{\mathrm{new}}(p)
 =\rho^2\widehat f_{\rho,\sigma}(p/\rho).
\end{equation}

We check that this graph calculation includes every moved point
of \(L_\rho\). Any initial point in the support of
\(G_{\rho,\sigma}\) has
\[
 p/\rho\in D_{1,\sigma},\qquad |q|<(C+2)\rho.
\]
Lemma~\ref{lem:source} says that every such point of \(L_\rho\)
is on the single graph just used. Conversely, that
entire graph over the support of \(d_{\rho,\sigma}\) lies
in \(\zeta=1\), by \eqref{eq:C}.
All these points remain in the same chart with fixed momentum.
If a point lies outside the support of \(G_{\rho,\sigma}\), then a
neighborhood of that point has \(G_{\rho,\sigma}=0\), hence the
Hamiltonian vector field vanishes there; the constant trajectory is therefore
the unique Hamiltonian trajectory through that point.  Such a point cannot
enter the support later.  The two chart supports are disjoint.  Consequently
the graph calculation above accounts for every point of the deformed torus
that can move.

For every moved point, the preceding graph calculation gives \eqref{eq:solu}
with \(v\in D_{1,\sigma}\subset\mathcal A\). By \eqref{eq:C},
\[
|p(\tau)|<2\rho,\qquad |q(\tau)|<C\rho
\]
uniformly for \(0\le\tau\le1\). Thus every moved point satisfies
\[
|p(\tau)|+|q(\tau)|<(C+2)\rho=O(\rho).
\]
Consequently, for sufficiently small \(\rho\), any new intersection with \(L_*\) lies in the local target branch described by Lemma~\ref{lem:target}. 
By
Lemma~\ref{lem:exact-correction} and Lemma~\ref{lem:graph},
\eqref{eq:difference} therefore yields precisely one intersection in each
\(D_{1,\sigma}\).
The original four intersections lie over \(D_\sigma\) for
small \(\rho\), by Proposition~\ref{prop:limit}; hence all
are included in this replacement.
Outside these regions the torus is unchanged, and
Proposition~\ref{prop:four} excludes any additional original
intersection. Every moved point stays within the charts where
\eqref{eq:difference} was counted, so no new intersection
elsewhere is possible.

To summarize the global exclusion explicitly: the complete
intersection set of \(L_\rho\) with \(L_*\) is the four-point
set of Proposition~\ref{prop:four}, and all four points lie in
the two correction disks for sufficiently small \(\rho\).
Every point of \(L_\rho\) that moves under \(\Psi_\rho\) lies in one of the local Darboux charts centered at \(P_\sigma\), where \(L_\rho\) is represented by the graph \(q=\nabla_p A_\rho(p)\), and its trajectory remains within that chart, whereas every
point outside the Hamiltonian support is fixed. Hence the two
local critical-point counts exhaust all global intersections of
\(\Psi_\rho(L_\rho)\) with \(L_*\); no uncounted intersection
can occur in the complement.

At either remaining critical point,
\(\widehat f_{\rho,\sigma}=\widetilde f_{0,\sigma}\).
In the coordinates of Lemma~\ref{lem:merge}, its Hessian
vanishes because its nonconstant part is \(t(t^2-s^2)\).
At a critical point, a change of variables transforms a
Hessian by congruence; thus it is zero in the original
momentum variables as well.
Equation \eqref{eq:difference} and Lemma~\ref{lem:graph}
prove equality of the tangent planes.
\end{proof}

\subsection*{A smooth periodic Hamiltonian and proof of the main theorem}

The inverse of the rotation \(\mathcal R\) in \eqref{eq:Rot}
is the time-one map of
\begin{equation}\label{eq:J}
 J(j,k)=-\frac{\pi}{2}(j_x+k_y).
\end{equation}
Choose nonnegative functions \(b_0,b_1,b_2\in C^\infty(\R/\mathbb Z)\)
whose supports are compactly contained, respectively, in
\((0,1/3)\), \((1/3,2/3)\) and \((2/3,1)\), and whose
integrals are all one. Such functions are obtained by translating
and rescaling a positive smooth bump and dividing by its integral.
Define the explicit Hamiltonian, with the scalar functions specified
in the preceding sections, by
\begin{equation}\label{eq:H}
 H(t,j,k)=b_0(t)K_\rho(j,k)+b_1(t)G_\rho(j,k)+b_2(t)J(j,k).
\end{equation}
It vanishes near the endpoints of the unit time interval and
therefore defines a smooth periodic function.

\begin{proof}[Proof of Theorem~\ref{thm:main}]
On each of the three time intervals, the flow in
\eqref{eq:H} is a time reparameterization of the respective
autonomous flow. The integral-one normalization of the
coefficients gives
\[
 \varphi_H^1=
 \mathcal R^{-1}\circ\Psi_\rho\circ \Phi_\rho.
\]
Since \(\mathcal R(L)=L_*\),
\[
 L\cap\varphi_H^1(L)
 =\mathcal R^{-1}\bigl(L_*\cap\Psi_\rho(L_\rho)\bigr).
\]
Proposition~\ref{prop:realization} proves that this set has
exactly two points and that their tangent planes coincide.
Rotating the canonical charts by the symplectic map
\(\mathcal R^{-1}\) preserves the potential description.
In the \((s,t)\) coordinates used in Lemma~\ref{lem:merge},
\eqref{eq:difference} is
\(\rho^2c+\rho^2t(t^2-s^2)\).
The further nonsingular change
\(s'=\rho^{2/3}s,\ t'=\rho^{2/3}t\)
makes its nonconstant part exactly \(t'(t'^2-s'^2)\).
This also proves that each intersection is isolated and gives
the local expression stated in the preceding remark; the vanishing Hessian
also recovers the nontransversality asserted in Theorem~\ref{thm:main}.
\end{proof}

\section{The Clifford torus in \texorpdfstring{\(\CP^2\)}{CP2}}\label{sec:cp2}

We now prove Theorem~\ref{thm:cp2}.  The argument reuses only the relative
saddle-combination lemma (cf.~Lemma~\ref{lem:merge}); the initial four-intersection configuration is
specific to \(\CP^2\).

Let
\[
 X^\circ=\{[z_0:z_1:z_2]\in\CP^2:z_0z_1z_2\neq0\}
\]
be the open toric orbit.  On \(X^\circ\), put
\begin{equation}\label{eq:cp2-action-angle}
 q_1=\arg(z_1/z_0),\qquad q_2=\arg(z_2/z_0),
\end{equation}
and
\[
 \mu_j([z])=\frac{|z_j|^2}{|z_0|^2+|z_1|^2+|z_2|^2},\qquad
 p_j=\frac13-\mu_j,\qquad j=1,2.
\]
Here the arguments are understood modulo \(2\pi\), so that \(q_1,q_2\in\mathbb R/2\pi\mathbb Z\). 
Thus \(q=(q_1,q_2)\in(\R/2\pi\mathbb Z)^2\) are angle variables and
\(p=(p_1,p_2)\) are the centered moment variables.  With our normalization
of \(\omega_{\mathrm{FS}}\), the standard toric action--angle formula is
\begin{equation}\label{eq:cp2-canonical}
 \omega_{\mathrm{FS}}
 =\mathrm dq_1\wedge\mathrm dp_1+
  \mathrm dq_2\wedge\mathrm dp_2.
\end{equation}
Indeed, on the affine chart $z_0\neq 0$, set
\[
w_j=\frac{z_j}{z_0}=r_je^{iq_j},
\qquad
D=1+|w_1|^2+|w_2|^2.
\]
With the normalization
$\int_{\CP^1}\omega_{\mathrm{FS}}=2\pi$,
we have
\[
\omega_{\mathrm{FS}}
=i\partial\bar\partial\log D
=dd^c\log D,
\qquad
d^c=\frac{i}{2}(\bar\partial-\partial).
\]
A direct computation gives
\[
\begin{aligned}
d^c\log D
&=\frac{i}{2D}
  \sum_{j=1}^2
  \bigl(w_j\,d\bar w_j-\bar w_j\,dw_j\bigr)\\
&=\sum_{j=1}^2\frac{r_j^2}{D}\,dq_j=\sum_{j=1}^2\mu_j\,dq_j.
\end{aligned}
\]
Consequently,
\[
\begin{aligned}
\omega_{\mathrm{FS}}
&=\sum_{j=1}^2 d\mu_j\wedge dq_j=\sum_{j=1}^2 dq_j\wedge dp_j,
\end{aligned}
\]
where the last equality follows from
$p_j=1/3-\mu_j$.

The Clifford torus is the zero section \(T_{\mathrm{Cl}}=\{p=0\}\) in these
coordinates.  The triangle
\begin{equation}\label{eq:cp2-moment-domain}
 \Delta_p^\circ=
 \{p_1<1/3,\ p_2<1/3,\ p_1+p_2>-1/3\}.
\end{equation} is the image of the open dense torus orbit under the moment map $p=(1/3-\mu_1,1/3-\mu_2)$. 

We shall use the position-graph version of Lemma~\ref{lem:graph}.

\begin{lemma}\label{lem:position-graph}
In canonical coordinates with \(\omega=\sum_i\mathrm dq_i\wedge\mathrm dp_i\),
a position graph
\[
 \Lambda_f=\{(q,p):p=\grad f(q)\}
\]
is Lagrangian.  Its intersections with the zero section are precisely the
critical points of \(f\).  Such an intersection is transverse if and only if
\(D^2f\) is invertible there; if \(D^2f=0\), the two tangent planes coincide.
\end{lemma}
\begin{proof}
The pullback of \(\omega\) to the graph is
\(\sum_{i,j}f_{ij}\,\mathrm dq_i\wedge\mathrm dq_j=0\).
The remaining assertions follow by comparing
\(T\Lambda_f=\{(\xi,D^2f\,\xi)\}\) with the horizontal tangent space of the
zero section.
\end{proof}

To obtain an explicit configuration suitable
for the saddle-merging construction of
Lemma~\ref{lem:merge}, we choose an
orthogonal transformation whose image of
the Clifford torus has four transverse
intersection points with the original torus,
organized into two pairs joined by
equal-level arcs. The following matrix,
with its special sum-and-difference
structure, allows these properties to be
verified explicitly. 
Consider the real orthogonal matrix
\begin{equation}\label{eq:U}
 U=
 \begin{pmatrix}
  1/\sqrt2&1/2&-1/2\\
  1/\sqrt2&-1/2&1/2\\
  0&1/\sqrt2&1/\sqrt2
 \end{pmatrix}.
\end{equation}
Although \(\det U=-1\), the matrix \(V=-U\) lies in
\(SO(3)\subset SU(3)\), and \(U\) and \(V\) induce the same projective
transformation.  We denote it by
\[
 \Phi_U([z])=[Uz].
\]
Then we have \[
\Phi_U\in\operatorname{Ham}
(\CP^2,\omega_{\mathrm{FS}})
.\]
Indeed, since $V\in SO(3)$, its rotation angle
$\theta\in(0,\pi)$ satisfies
\[
\cos\theta
=\frac{\operatorname{tr}(V)-1}{2}
=-\frac{1+2\sqrt2}{4}.
\]
By Rodrigues' rotation formula
\cite{Rodrigues1840,LynchPark2017}, there is
a skew-symmetric matrix $K$ such that
\[
V=I+\sin\theta\,K
  +(1-\cos\theta)K^2,
\qquad
K^3=-K.
\]
Taking the transpose $V^T$ of $V$ gives
\[
V-V^T=2\sin\theta\,K.
\]
Hence
\[
B:=\theta K
=\frac{\theta}{2\sin\theta}(V-V^T)
\in\mathfrak{so}(3).
\]
Using $K^3=-K$, we obtain
\[
e^B=e^{\theta K}
=I+\sin\theta\,K
 +(1-\cos\theta)K^2
=V.
\]
Set \(A=iB\), a Hermitian trace-free matrix, and define
\begin{equation}\label{eq:cp2-HU}
 H_U([z])=\frac{z^*Az}{z^*z},
\end{equation}
where \(z^*\) denotes conjugate transpose.

The angle $\theta$ and the matrix $B$
are determined entirely by $V$ and are
independent of the normalization of the
Fubini--Study symplectic form. With the normalization
$\int_{\CP^1}\omega_{\mathrm{FS}}=2\pi$
and the Hamiltonian sign convention
\eqref{eq:ham-convention}, the standard
moment-map formula for the projective
unitary action gives
\begin{equation}\label{eq:cp2-HU-flow}
 \varphi_{H_U}^t([z])
 =[e^{-itA}z]=[e^{tB}z].
\end{equation}

To verify the flow formula, first suppose
that $A=\operatorname{diag}(\lambda_0,
\lambda_1,\lambda_2)$ is diagonal.
Using $\mu_0+\mu_1+\mu_2=1$ and
$p_j=1/3-\mu_j$, we obtain
\[
H_U
=\frac{\lambda_0+\lambda_1+\lambda_2}{3}
-(\lambda_1-\lambda_0)p_1
-(\lambda_2-\lambda_0)p_2.
\]
By \eqref{eq:cp2-canonical} and the
Hamiltonian convention
\eqref{eq:ham-convention}, Hamilton's
equations are
\[
\dot q_j=-(\lambda_j-\lambda_0),
\qquad
\dot p_j=0,\qquad j=1,2.
\]
Since $q_j=\arg(z_j/z_0)$, this gives
\[
\varphi_{H_U}^t([z])
=[e^{-it\lambda_0}z_0:
  e^{-it\lambda_1}z_1:
  e^{-it\lambda_2}z_2]
=[e^{-itA}z].
\]
For a general Hermitian matrix, write
$A=WDW^*$, where $W$ is unitary and
$D$ is real diagonal. Since $\Phi_W$
preserves $\omega_{\mathrm{FS}}$ and
$H_U=H_D\circ\Phi_{W^*}$, the corresponding
flows are conjugate:
\[
\varphi_{H_U}^t
=\Phi_W\circ\varphi_{H_D}^t
 \circ\Phi_{W^*}.
\]
Consequently,
\[
\varphi_{H_U}^t([z])
=[We^{-itD}W^*z]
=[e^{-itA}z]
=[e^{tB}z].
\]
In particular, since $e^B=V=-U$,
\[
\varphi_{H_U}^1=\Phi_V=\Phi_U.
\]

\begin{proposition}\label{prop:cp2-four}
The Lagrangians \(T_{\mathrm{Cl}}\) and \(\Phi_U(T_{\mathrm{Cl}})\) intersect
transversely in exactly four points, namely
\begin{equation}\label{eq:cp2-four-points}
 P_{\varepsilon,\delta}
 =\bigl[e^{i\varepsilon\pi/4}:e^{-i\varepsilon\pi/4}:\delta\bigr],
 \qquad \varepsilon,\delta\in\{+1,-1\}.
\end{equation}
\end{proposition}
\begin{proof}
Every point of \(T_{\mathrm{Cl}}\) has a unique representative of the form
\([
 1:e^{i\alpha}:e^{i\beta}
])\).  Set
\(y=U(1,e^{i\alpha},e^{i\beta})^T\).  Its components are
\begin{equation}\label{eq:cp2-y}
  y_0=\frac1{\sqrt2}+\frac{e^{i\alpha}-e^{i\beta}}2,\qquad
  y_1=\frac1{\sqrt2}-\frac{e^{i\alpha}-e^{i\beta}}2,\qquad
  y_2=\frac{e^{i\alpha}+e^{i\beta}}{\sqrt2}.
\end{equation}
Because \(U\) is unitary, \(|y_0|^2+|y_1|^2+|y_2|^2=3\).  Hence
\([y]\in T_{\mathrm{Cl}}\) if and only if
\(|y_0|=|y_1|\) and \(|y_2|=1\).  Directly from
\eqref{eq:cp2-y}, these conditions are
\begin{equation}\label{eq:cp2-intersection-equations}
 \cos\alpha=\cos\beta,\qquad
 \cos(\alpha-\beta)=0.
\end{equation}
The first equation gives \(\beta\equiv\pm\alpha\pmod{2\pi}\); the plus sign
is incompatible with the second equation.  Thus $\beta\equiv-\alpha\pmod{2\pi}$ and
$\cos 2\alpha=0$, yielding exactly four
solutions modulo $2\pi$:
\[
(\alpha,\beta)\in
\left\{
\left(\frac{\pi}{4},\frac{7\pi}{4}\right),
\left(\frac{3\pi}{4},\frac{5\pi}{4}\right),
\left(\frac{5\pi}{4},\frac{3\pi}{4}\right),
\left(\frac{7\pi}{4},\frac{\pi}{4}\right)
\right\}.
\]
Substituting these four solutions into \eqref{eq:cp2-y} gives
\eqref{eq:cp2-four-points}.

Consider the parametrization
\[
\Psi:\mathbb T^2\longrightarrow \Phi_U(T_{\mathrm{Cl}}),
\qquad
(\alpha,\beta)\longmapsto [y],
\]
where $y$ is given in \eqref{eq:cp2-y}. 
 We now define
\[
\begin{aligned}
F_1(\alpha,\beta)
&=|y_0|^2-|y_1|^2=\sqrt{2}(\cos\alpha-\cos\beta),\\[4pt]
F_2(\alpha,\beta)
&=|y_2|^2-1=\cos(\alpha-\beta).
\end{aligned}
\]
Then
\(
\Psi(\alpha,\beta)\in T_{\mathrm{Cl}}\)
if and only if 
\(F_1(\alpha,\beta)=F_2(\alpha,\beta)=0.
\)
At a solution with \(\beta=-\alpha\),
\[
 \det\frac{\partial(F_1,F_2)}{\partial(\alpha,\beta)}
 =-2\sqrt2\,\sin\alpha\,\sin2\alpha\neq0.
\]

Since $\|y\|^2=3$ and $p_j=1/3-\mu_j$, we have
\[
\frac{F_1}{3}=2p_1+p_2,
\qquad
\frac{F_2}{3}=-p_2.
\]
The corresponding linear transformation is invertible.
Thus, the nonsingularity of the Jacobian of
$(F_1,F_2)$ implies that the restriction of
$p=(p_1,p_2)$ to $\Phi_U(T_{\mathrm{Cl}})$
has invertible differential at each intersection point.
Since $T_{\mathrm{Cl}}=\{p=0\}$, this proves that
$T_{\mathrm{Cl}}$ and $\Phi_U(T_{\mathrm{Cl}})$
intersect transversely.
\end{proof}

The pairing of the four points has a simple symmetry.  Let
\[
 \mathcal S[z_0:z_1:z_2]=[z_1:z_0:z_2],\qquad
 \mathcal T[z_0:z_1:z_2]=[z_0:z_2:z_1].
\]
Both projective coordinate permutations preserve \(T_{\mathrm{Cl}}\), and
the permutation matrices \(P_{\mathcal S}\) and \(P_{\mathcal T}\), representing these two coordinate permutations on \(\mathbb C^3\), satisfy \(P_{\mathcal S}U=UP_{\mathcal T}\).  Consequently
\(\mathcal S\circ\Phi_U=\Phi_U\circ \mathcal T\), while
\(\mathcal S(P_{\varepsilon,\delta})=P_{-\varepsilon,\delta}\).  Thus the four
intersections are naturally organized into the two pairs with fixed
\(\delta\).  This finite symmetry plays the organizational role suggested by
the symmetric-orbit viewpoint of \cite{LuAG}; it should not be confused with
a real anti-symplectic involution, which does not exist for the Clifford torus
in \(\CP^2\).

\begin{figure}[htbp]
    \centering
    \includegraphics[width=0.58\textwidth]{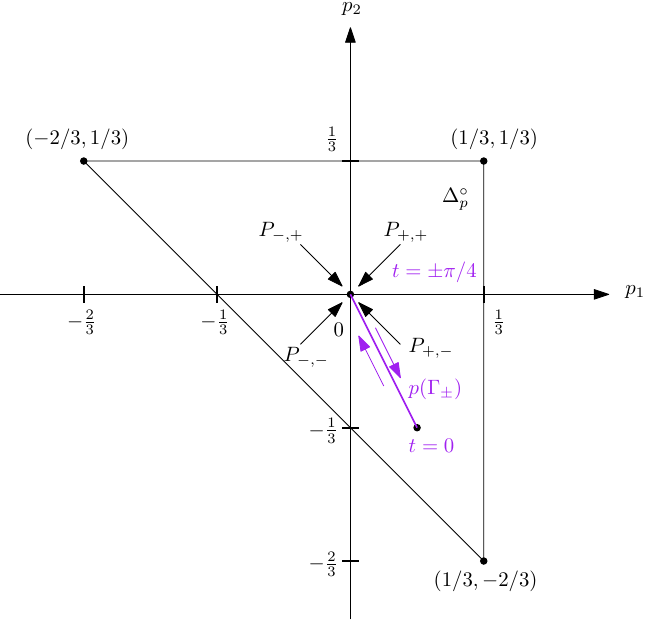}
    \caption{The image of the arcs $\Gamma_\pm$  with four endpoints $P_{\varepsilon,\delta}$ having the same image $0$ under the moment map $p$.}
    \label{fig:triangle}
\end{figure}

The four points split into two pairs joined by exact level arcs.  For
\(|t|<\pi/2\), define
\begin{align}
 \Gamma_+(t)&=\Phi_U[1:e^{it}:e^{-it}],\label{eq:Gamma-plus}\\
 \Gamma_-(t)&=\Phi_U[1:-e^{it}:-e^{-it}].\label{eq:Gamma-minus}
\end{align}
Then
\begin{align}
 \Gamma_+(t)&=
 \left[\frac1{\sqrt2}+i\sin t:
       \frac1{\sqrt2}-i\sin t:
       \sqrt2\cos t\right],\label{eq:Gamma-plus-explicit}\\
 \Gamma_-(t)&=
 \left[\frac1{\sqrt2}-i\sin t:
       \frac1{\sqrt2}+i\sin t:
       -\sqrt2\cos t\right].\label{eq:Gamma-minus-explicit}
\end{align}
The first arc passes through \(P_{-,+}\) and \(P_{+,+}\) at
\(t=-\pi/4\) and \(t=\pi/4\), respectively; the second passes through
\(P_{+,-}\) and \(P_{-,-}\) at those two parameter values; see Figure~\ref{fig:triangle}.

\begin{lemma}\label{lem:cp2-level-arcs}
For each sign \(\sigma\in\{+,-\}\), a neighborhood of the compact subarc of
\(\Gamma_\sigma\) obtained by extending slightly beyond
\([-\pi/4,\pi/4]\) is a single position graph
\[
 p=\grad f_\sigma(q)
\]
over a disk in the angle variables.  The two endpoints corresponding to
\(t=\pm\pi/4\) are nondegenerate saddles of \(f_\sigma\), have the same
critical value, and are joined by an embedded regular arc in that common
level set.  The two graph neighborhoods can be chosen disjoint, and each
contains no other critical point of its generating function.
\end{lemma}
\begin{proof}
Put
\[
R(t)=\sqrt{\frac12+\sin^2t},
\qquad
a(t)=\arctan(\sqrt2\sin t).
\]
The explicit parametrizations of the two arcs can
then be written as
\[
\begin{aligned}
\Gamma_+(t)
&=\bigl[R(t)e^{ia(t)}:
         R(t)e^{-ia(t)}:
         \sqrt2\cos t\bigr],\\
\Gamma_-(t)
&=\bigl[R(t)e^{-ia(t)}:
         R(t)e^{ia(t)}:
         -\sqrt2\cos t\bigr].
\end{aligned}
\]
Since $\cos t>0$ along both arcs, taking
the arguments of $y_1/y_0$ and $y_2/y_0$
gives the angular coordinates. Moreover,
\[
|y_1|^2=\frac12+\sin^2t,\qquad
|y_2|^2=2\cos^2t,\qquad
\|y\|^2=3.
\]
Thus, using $p_j=1/3-|y_j|^2/3$ and \eqref{eq:cp2-action-angle}, we obtain,
along $\Gamma_+$,
\begin{equation}\label{eq:cp2-plus-coordinates}
 q_1=-2a(t),\qquad q_2=-a(t),\qquad
 p_1=\frac{\cos2t}{6},\qquad
 p_2=-\frac{\cos2t}{3},
\end{equation}
and, along $\Gamma_-$,
\begin{equation}\label{eq:cp2-minus-coordinates}
 q_1=2a(t),\qquad q_2=\pi+a(t),\qquad
 p_1=\frac{\cos2t}{6},\qquad
 p_2=-\frac{\cos2t}{3}.
\end{equation}
Here the angle coordinates are understood through
the indicated continuous lifts to $\mathbb R^2$.
The additional $\pi$ in $q_2$ along $\Gamma_-$
arises from the negative sign of its third
homogeneous coordinate. In either case, a direct computation shows that 
\begin{equation}\label{eq:cp2-action-zero}
 p_1\,\mathrm dq_1+p_2\,\mathrm dq_2=0
 \qquad\text{along }\Gamma_\sigma.
\end{equation}

We next verify that the angle projection is regular on the entire extended
arc.  Parameterize \(\Phi_U(T_{\mathrm{Cl}})\) by \((\alpha,\beta)\) through
\eqref{eq:cp2-y}.  
To compute the Jacobian of the angular projection,
introduce
\[
m=\frac{\alpha+\beta}{2},
\qquad
d=\frac{\alpha-\beta}{2}.
\]
Then
\[
y_0=\frac1{\sqrt2}+ie^{im}\sin d,\qquad
y_1=\frac1{\sqrt2}-ie^{im}\sin d,\qquad
y_2=\sqrt2e^{im}\cos d.
\]
Along the two arcs, we have $d=t$ and
$e^{im}=\sigma$, where $\sigma=1$ on
$\Gamma_+$ and $\sigma=-1$ on $\Gamma_-$.
Set
\[
D=1+2\sin^2t=2-\cos2t.
\]
Using
$d\arg w=\operatorname{Im}(dw/w)$
and
\[
q_1=\arg y_1-\arg y_0,\qquad
q_2=\arg y_2-\arg y_0,
\]
we obtain
\[
\frac{\partial(q_1,q_2)}{\partial(m,d)}
=
\begin{pmatrix}
0&-\dfrac{2\sigma\sqrt2\cos t}{D}\\[6pt]
\dfrac1D&-\dfrac{\sigma\sqrt2\cos t}{D}
\end{pmatrix}.
\]
Since
\[
\det\frac{\partial(m,d)}
{\partial(\alpha,\beta)}=-\frac12,
\]
the chain rule yields
\[
\det\frac{\partial(q_1,q_2)}
{\partial(\alpha,\beta)}
=
-\frac{\sigma\sqrt2\cos t}{D^2}.
\]
Consequently,
\begin{align}
\det\frac{\partial(q_1,q_2)}
{\partial(\alpha,\beta)}
\bigg|_{(\alpha,\beta)=(t,-t)}
&=-\frac{\sqrt2\cos t}
{(2-\cos2t)^2},
\label{eq:cp2-jac-plus}\\
\det\frac{\partial(q_1,q_2)}
{\partial(\alpha,\beta)}
\bigg|_{(\alpha,\beta)=(\pi+t,\pi-t)}
&=\frac{\sqrt2\cos t}
{(2-\cos2t)^2}.
\label{eq:cp2-jac-minus}
\end{align}
Both determinants are nonzero for $|t|<\pi/2$.

Although the nonvanishing Jacobian guarantees
that the angular projection is locally
invertible along each arc, it does not
exclude other points of
$\Phi_U(T_{\mathrm{Cl}})$ having the same
angular coordinates. To obtain a single-valued
position graph over a neighborhood of the
entire compact arc, we must rule out such
additional sheets. We now verify this explicitly.
  Put
\(A=e^{i\alpha}\) and \(B=e^{i\beta}\).  Whenever \(y_2\neq0\), the
relation \(q_1-2q_2\equiv0\pmod{2\pi}\) is equivalent to
\(y_0y_1/y_2^2\) being positive real.  Writing
\(m=(\alpha+\beta)/2\) and \(d=(\alpha-\beta)/2\), one computes from
\eqref{eq:cp2-y}
\begin{equation}\label{eq:cp2-phase-ratio}
 \frac{y_0y_1}{y_2^2}
 =\frac{e^{-2im}+2\sin^2d}{4\cos^2d}.
\end{equation}
Thus its imaginary part can vanish only when \(AB=e^{2im}=\pm1\).
If \(AB=-1\), then \(B=-\overline A\); in that case \(y_0,y_1\) are real
and \(y_2\) is purely imaginary, so \(q_2\equiv\pm\pi/2\).  Such points do
not lie over either compact angle arc under consideration.  Indeed, along the compact angle arcs of
$\Gamma_+$ and $\Gamma_-$, respectively,
\[
q_2\in\left[-\frac{\pi}{4},\frac{\pi}{4}\right],
\qquad
q_2\in\left[\frac{3\pi}{4},\frac{5\pi}{4}\right]
\pmod{2\pi}.
\]
Neither interval contains $\pm\pi/2$ modulo
$2\pi$. The same exclusion holds for
sufficiently short extensions of these arcs. 
If \(AB=1\), since $|A|=1$ we may parametrize the pairs $(A,B)$ by 
 \(A=e^{it}\) and  \(B=e^{-it}\), this is precisely the circle
\([1:e^{it}:e^{-it}]\).  On the component with \(\cos t>0\),
\(q_2=-a(t)\), while on the component with \(\cos t<0\),
\(q_2=\pi-a(t)\), with \(a(t)=\arctan(\sqrt2\sin t)\).  Since
\(a'(t)=\sqrt2\cos t/(1+2\sin^2t)\), each of these restrictions of the angular projection
\[\pi_q: X^\circ\longrightarrow (\R/2\pi\mathbb Z)^2, \quad(q,p)\longmapsto (q_1,q_2)\]
to $\Gamma_+$ and $\Gamma_-$
is injective.
Hence the full preimage \emph{in the open toric orbit} of either
compact angle arc is exactly the corresponding \(\Gamma_\sigma\)-arc, namely $$\pi_q^{-1}(\pi_q(\Gamma_\sigma))\cap \Phi_U(T_{\mathrm{Cl}})=\Gamma_\sigma.$$
To pass from this fiberwise statement to a genuine single-sheet
neighborhood, we must also rule out an extra sheet approaching a toric
divisor, where the angle coordinates cease to be defined.  

Choose sufficiently short extensions of the
two arcs, with $|t|\le \pi/4+\eta<\pi/2$.
Since
\[
|a(t)|<\arctan\sqrt2<\frac{\pi}{2},
\qquad
2\arctan\sqrt2=\arccos(-1/3),
\]
the formulas for $q_1$ and $q_2$ imply
\begin{equation}\label{eq:q1}
|q_1|<\arccos(-1/3),
\end{equation}
and that $q_2$ lies in $(-\pi/2,\pi/2)$
along $\Gamma_+$ and in $(\pi/2,3\pi/2)$
along $\Gamma_-$.
Finally,
\[
q_1-q_2=
\begin{cases}
-a(t), & \text{on }\Gamma_+,\\
a(t)-\pi, & \text{on }\Gamma_-,
\end{cases}
\]
so $q_1-q_2$ stays uniformly away from
$\pm\pi/2$ modulo $2\pi$.  These conditions hold at the original compact
segments \(|t|\le\pi/4\), with strict margins, and hence persist
under sufficiently short extensions. 
These angular bounds also exclude any toric-boundary accumulation of points in \(\Phi_U(T_{\mathrm{Cl}})\) whose angular coordinates approach either compact arc. Indeed, on the parameter torus the only possible toric-boundary limits are \(y_0=0\), \(y_1=0\), or \(y_2=0\).  If \(y_0=0\),
\eqref{eq:cp2-y} gives \(y_1=\sqrt2\) and \(y_2=\pm i\).
Along any sequence in \(X^\circ\) approaching such a point,
\(q_1-q_2=\arg(y_1/y_2)\) tends to \(\pm\pi/2\), contrary to
the chosen arc bounds.  If \(y_1=0\), then \(y_0=\sqrt2\) and
\(y_2=\pm i\), so \(q_2\to\pm\pi/2\), again impossible. 
Finally, if $y_2=0$, then $B=-A$, so
\[
y_0=\frac1{\sqrt2}+A,\qquad
y_1=\frac1{\sqrt2}-A.
\]
Since $|A|=1$, we have
\[
\operatorname{Re}(y_1\overline{y_0})
=\operatorname{Re}\left(
-\frac12+\frac{\overline A-A}{\sqrt2}
\right)
=-\frac12.
\]
Moreover, as $\|y\|^2=3$ and $y_2=0$,
\[
2|y_0y_1|
\le |y_0|^2+|y_1|^2=3.
\]
Consequently,
\[
\cos q_1
=\frac{\operatorname{Re}(y_1\overline{y_0})}
{|y_0y_1|}
=-\frac{1}{2|y_0y_1|}
\le-\frac13,
\]
contradicting the chosen bound
$|q_1|<\arccos(-1/3)$ (see~\eqref{eq:q1}). 
Thus no sequence of source points with angles approaching either compact
arc can accumulate on the toric boundary.

Now suppose that arbitrarily small angle neighborhoods of one compact
arc contained a source point on a different local sheet.  We could choose
such points with angles tending to the compact arc.  Compactness of
\(\Phi_U(T_{\mathrm{Cl}})\), together with the boundary exclusion above,
would yield an extra point in the open toric orbit over that arc,
contrary to the preceding explicit preimage calculation and the inverse
function theorem at \eqref{eq:cp2-jac-plus}--\eqref{eq:cp2-jac-minus}.
Consequently sufficiently small angle neighborhoods are covered by
exactly one sheet of \(\Phi_U(T_{\mathrm{Cl}})\), and the angle
projection identifies this sheet with a disk.  The two disks can be
chosen disjoint, since their \(q_2\)-coordinates are near \(0\) and
\(\pi\), respectively.

Since \(\Phi_U(T_{\mathrm{Cl}})\) is Lagrangian, the one-form
\(p_1\,\mathrm dq_1+p_2\,\mathrm dq_2\) on either graph is closed.  On a
disk it is therefore \(\mathrm df_\sigma\) for a smooth function
\(f_\sigma\).  Equation \eqref{eq:cp2-action-zero} shows that
\(f_\sigma\) is constant along \(\Gamma_\sigma\), including a short
extension beyond both intersection points.  The endpoints are critical
points because \(p=0\) there.  They are nondegenerate by
Proposition~\ref{prop:cp2-four} and Lemma~\ref{lem:position-graph}.
A nondegenerate extremum cannot lie on a nonconstant smooth arc in its own
critical level set; hence both critical points are saddles.  Finally, the
four intersections listed in Proposition~\ref{prop:cp2-four} are the only
ones globally, so the two disks can be chosen to contain no other critical
point.
\end{proof}

We now combine each same-level saddle pair.
For each $\sigma\in\{+,-\}$, let $f_\sigma$
be the local generating function obtained
in Lemma~\ref{lem:cp2-level-arcs}.
Choose a slightly smaller disk $E_\sigma$,
still containing the extended level arc.
Applying Lemma~\ref{lem:merge} to
$f=f_\sigma$, we obtain a modified function
$\widetilde f_\sigma$ by merging its two
saddle critical points into one.
We require the following additional
moment-domain control.

\medskip
\noindent\textit{Moment-domain control.}
The modification can be chosen so that, in addition to the conclusions of
Lemma~\ref{lem:merge},
\begin{equation}\label{eq:cp2-moment-control}
 \mathrm d\widetilde f_\sigma(q)\in\Delta_p^\circ
 \quad\text{and}\quad
 (1-\tau)\mathrm df_\sigma(q)+
 \tau\mathrm d\widetilde f_\sigma(q)\in\Delta_p^\circ
\end{equation}
for every \(q\in\overline{E_\sigma}\) and \(0\le\tau\le1\).
We give the control explicitly.  Since
\(\mathrm df_\sigma(\overline{E_\sigma})\subset\Delta_p^\circ\) and
\(0\in\Delta_p^\circ\), the compact set
\[
 \mathcal P_\sigma=
 \{\lambda\,\mathrm df_\sigma(q):
 q\in\overline{E_\sigma},\ 0\le\lambda\le1\}
\]
is contained in the open convex triangle \(\Delta_p^\circ\).  Hence there is
\(\nu_\sigma>0\) such that the closed \(\nu_\sigma\)-neighborhood of
\(\mathcal P_\sigma\) is still contained in \(\Delta_p^\circ\).

Fix once and for all the local coordinate
map $\Theta_\sigma:q\mapsto(s,t)$ obtained
from Steps~1--2 of Lemma~\ref{lem:merge},
and denote its inverse by
$\Psi_\sigma=\Theta_\sigma^{-1}$.
Thus,
\[
f_\sigma(\Psi_\sigma(s,t))=c_\sigma+t(t^2+1-s^2).
\]
Choose fixed numbers \(\eta_*>0\) and \(r_*>0\) such that the
closed rectangle
\[
 \mathcal R_{*,\sigma}=
 \{(s,t): |s|\le 1+\eta_*,\ |t|\le r_*\}
\]
lies in this normalized chart and its image is compactly contained
in \(E_\sigma\).  After subtracting the common critical value, we obtain the function
\[
 F_0(s,t)=t(t^2+1-s^2)
\]
on this fixed rectangle.  All later parameters satisfy
\(0<\eta<\eta_*\) and \(0<r_1<r_*\); the coordinate map and the
rectangle \(\mathcal R_{*,\sigma}\) do \emph{not} depend on them.
For an invertible matrix \(B\), write
\(B^{-T}=(B^{-1})^T\), and use the Euclidean operator norm.
In particular, the finite constant
\[
 M_{*,\sigma}:=\sup_{x\in\mathcal R_{*,\sigma}}
 \bigl\|(D\Psi_\sigma(x))^{-T}\bigr\|
\]
is fixed before any cutoff or scale is chosen.  Choose
\(0<\delta^2\ll\eta\), and choose an even smooth function
\(a_\delta(s)\) such that
\begin{equation}\label{eq:cp2-adelta}
 \begin{aligned}
 &a_\delta(s)=-\delta^2s^2 &&\text{for } |s|\le\frac12,\\
 &a_\delta(s)=1-s^2 &&\text{for } |s|\ge1+\eta,\\
 &a_\delta(s)<0,\qquad s a_\delta'(s)<0 &&\text{for } s\ne0.
 \end{aligned}
\end{equation}
We give an explicit existence argument for this interpolation,
including the uniform size bound. Fix \(\eta>0\) first and take
\(\delta>0\) so small that \(\delta^2/4<2\eta+\eta^2\).
For \(s\ge0\), construct a smooth function \(g(s)\) that is
positive for \(s>0\), equals \(2\delta^2s\) on a neighborhood
of \([0,1/2]\), equals \(2s\) on a neighborhood of
\([1+\eta,\infty)\), and satisfies
\begin{equation}\label{eq:cp2-g-integral}
 \int_{1/2}^{1+\eta}g(s)\dd s
 =2\eta+\eta^2-\frac{\delta^2}{4}.
\end{equation}
Here is a direct construction. On arbitrarily short collars of the
two endpoints retain the prescribed positive functions, and smoothly
connect them by a strictly positive background whose integral over
\([1/2,1+\eta]\) is smaller than the positive right-hand side
of \eqref{eq:cp2-g-integral}. This is possible by first shrinking
the endpoint collars and then choosing the interior background
small. Add a nonnegative smooth bump compactly supported in
\((1/2,1+\eta)\), normalized to have integral one, multiplied
by the positive difference between the required integral and the
background integral. The resulting \(g\) is smooth, strictly
positive away from zero, and satisfies all three conditions.
Define
\[
 a_\delta(s)=-\int_0^s g(u)\dd u\quad(s\ge0),
 \qquad a_\delta(-s)=a_\delta(s).
\]
Near the origin it equals \(-\delta^2s^2\), so the even extension
is smooth. The integral condition gives
\(a_\delta(1+\eta)=-2\eta-\eta^2=1-(1+\eta)^2\);
as \(g(s)=2s\) for \(s\ge1+\eta\), this yields
\(a_\delta(s)=1-s^2\) there. Positivity of \(g\) gives
\(a_\delta(s)<0\) and \(s a_\delta'(s)<0\) for \(s\ne0\),
and monotonicity yields the quantitative bound
\begin{equation}\label{eq:cp2-adelta-bound}
 \sup_{|s|\le1+\eta}|a_\delta(s)|
 =2\eta+\eta^2.
\end{equation}
This construction is completed before choosing \(r_1,r_0\),
so its finite derivative bounds cannot depend on those scales.
Put
\[
 F_\delta(s,t)=t\bigl(t^2+a_\delta(s)\bigr).
\]
Next choose \(0<r_0<r_1\) and an even cutoff
\(\beta\in C_c^\infty((-r_1,r_1),[0,1])\) which is one on
\(|t|\le r_0\), is nonincreasing in \(|t|\), and satisfies
\begin{equation}\label{eq:cp2-log-cutoff}
 \sup_t |t\beta'(t)|<\gamma,
\end{equation}
where \(\gamma>0\) can be made arbitrarily small.  For example, take a fixed
smooth nonincreasing \(\chi\), equal to one on \(( -\infty,0]\) and zero on
\([1,\infty)\), and on \(r_0<|t|<r_1\) set
\[
 \beta(t)=\chi\!\left(
 \frac{\log(|t|/r_0)}{\log(r_1/r_0)}\right).
\]
Then \(\sup|t\beta'|\le
\|\chi'\|_\infty/\log(r_1/r_0)\), so
\eqref{eq:cp2-log-cutoff} follows by taking \(r_0/r_1\) sufficiently small.
Define
\begin{equation}\label{eq:cp2-controlled-merger}
 \widetilde F
 =F_0+\beta(t)(F_\delta-F_0).
\end{equation}
Because \(a_\delta=1-s^2\) for \(|s|\ge1+\eta\) and \(\beta=0\) for
\(|t|\ge r_1\), the difference \(\widetilde F-F_0\) is compactly supported
in the normalized rectangle.

We first check that this controlled modification has exactly one critical
point.  
On the normalized rectangle
$\mathcal R_{*,\sigma}$, differentiation gives
\[
\partial_s\widetilde F
=t\bigl(-2s(1-\beta(t))
       +\beta(t)a_\delta'(s)\bigr).
\]
Since $0\le\beta\le1$ and
$s a_\delta'(s)<0$ for $s\ne0$,
the factor in parentheses has sign
$-\operatorname{sgn}(s)$ whenever $s\ne0$.
Consequently, any critical point with
$t\ne0$ must satisfy $s=0$.
  Since \(a_\delta(0)=0\), at \(s=0\) we have
\[
 \partial_t\widetilde F(0,t)
 =3t^2+1-\beta(t)-t\beta'(t)>0
 \qquad (t\ne0),
\]
because \(t\beta'(t)\le0\).  On the other hand, at \(t=0\),
\(\partial_t\widetilde F(s,0)=a_\delta(s)\), which vanishes only at
\(s=0\).  Thus the origin is the unique critical point in the modification rectangle.
Outside that rectangle \(\widetilde F=F_0\), and
Lemma~\ref{lem:cp2-level-arcs} excludes any further critical point in
\(E_\sigma\).  Near the origin \(\beta=1\) and
\(a_\delta=-\delta^2s^2\), so
\[
 \widetilde F=t(t^2-\delta^2s^2),
\]
which becomes \(t(t^2-s'^2)\) after the nonsingular change
\(s'=\delta s\).

It remains to verify the fiber control.  Let
\(C_a=\sup_{|s|\le1+\eta}|a_\delta'(s)|\) and
\(C_D=\sup_{|s|\le1+\eta}|a_\delta(s)-(1-s^2)|\).  On the support of the
modification,
\[
 \mathrm d\widetilde F
 =(1-\beta)\mathrm dF_0+\beta\mathrm dF_\delta
 +(F_\delta-F_0)\beta'(t)\,\mathrm dt.
\]
For \(|s|\le1+\eta\) and \(|t|\le r_1\),
\[
 \|\mathrm dF_\delta\|
 \le r_1C_a+3r_1^2+2\eta+\eta^2,
\]
whereas for \(|s|\ge1+\eta\) one has
\(\mathrm dF_\delta=\mathrm dF_0\).  Moreover,
\[
 \|(F_\delta-F_0)\beta'(t)\,\mathrm dt\|
 \le C_D\,|t\beta'(t)|\le C_D\gamma.
\]
Consequently
\begin{equation}\label{eq:cp2-normalized-distance}
 \operatorname{dist}\!\left(
 \mathrm d\widetilde F,
 \operatorname{conv}\{0,\mathrm dF_0\}
 \right)
 \le r_1C_a+3r_1^2+2\eta+\eta^2+C_D\gamma.
\end{equation}
At $q=\Psi_\sigma(x)$, the chain rule
applied to
$F_0=f_\sigma\circ\Psi_\sigma-c_\sigma$
gives
\[
\mathrm df_\sigma(q)
=(D\Psi_\sigma(x))^{-T}
\mathrm dF_0(x).
\]
Since $(D\Psi_\sigma(x))^{-T}$ is linear,
it follows that
\[
\begin{aligned}
&(D\Psi_\sigma(x))^{-T}
\bigl(\operatorname{conv}
\{0,\mathrm dF_0(x)\}\bigr)=
\{\lambda\,\mathrm df_\sigma(q):
0\le\lambda\le1\}
\subset\mathcal P_\sigma.
\end{aligned}
\]  The parameters can therefore be chosen without
circular dependence: first take \(\eta\in(0,\eta_*)\) so small that
\(2\eta+\eta^2<\nu_\sigma/(8M_{*,\sigma})\);
next fix \(\delta>0\) with \(\delta^2\ll\eta\) and construct
\(a_\delta\) as above (thereby fixing the finite constants
\(C_a,C_D\)); next take \(r_1\in(0,r_*)\) so small that
\(r_1C_a+3r_1^2<\nu_\sigma/(8M_{*,\sigma})\);
finally choose \(r_0/r_1\) sufficiently small to achieve
\(C_D\gamma<\nu_\sigma/(8M_{*,\sigma})\).
The right-hand side of \eqref{eq:cp2-normalized-distance} is then
strictly less than \(\nu_\sigma/(2M_{*,\sigma})\) throughout the
modification region.  Thus the subsequent pullback estimate uses a
uniform, previously fixed coordinate constant.  Pulling \(\widetilde F\) back to
\(E_\sigma\), restoring the critical value, and extending by \(f_\sigma\)
outside the modification rectangle therefore gives a function
\(\widetilde f_\sigma\) satisfying
\begin{equation}\label{eq:cp2-merger}
 \supp(\widetilde f_\sigma-f_\sigma)\Subset E_\sigma,
 \qquad
 \#\Crit(\widetilde f_\sigma|_{E_\sigma})=1,
\end{equation}
whose unique critical point has local form \(c+t(t^2-s^2)\), and whose
differential lies in \(\Delta_p^\circ\).  The second assertion in
\eqref{eq:cp2-moment-control} then follows immediately from the convexity of
\(\Delta_p^\circ\), because both endpoints
\(\mathrm df_\sigma(q)\) and \(\mathrm d\widetilde f_\sigma(q)\) lie in
that triangle.

Set
\[
 h_\sigma=\widetilde f_\sigma-f_\sigma.
\]
The compact interpolation set
\begin{equation}\label{eq:cp2-interpolation-set}
 \mathcal C_\sigma=
 \bigl\{(q,\mathrm df_\sigma(q)+
 \tau\,\mathrm dh_\sigma(q)):
 q\in\supp h_\sigma,\ 0\le\tau\le1\bigr\}
\end{equation}
lies in \(X^\circ\) by \eqref{eq:cp2-moment-control}.  Choose
\(\chi_\sigma\in C_c^\infty(X^\circ,[0,1])\) equal to one on a neighborhood
of \(\mathcal C_\sigma\), with support contained over the corresponding
angle disk.  By the single-sheet statement in
Lemma~\ref{lem:cp2-level-arcs}, these two supports can be chosen disjoint and
meet \(\Phi_U(T_{\mathrm{Cl}})\) only in the two graph patches under
consideration.  Define, on the toric chart and then by zero on its complement,
\begin{equation}\label{eq:cp2-Ksigma}
 K_\sigma(q,p)=-\chi_\sigma(q,p)h_\sigma(q),
 \qquad K=K_++K_-.
\end{equation}

\begin{proposition}\label{prop:cp2-two}
The Hamiltonian image \(\varphi_K^1\Phi_U(T_{\mathrm{Cl}})\) meets
\(T_{\mathrm{Cl}}\) in exactly two points, one in each local modification disk $E_\sigma$.  Both
are isolated and have coincident tangent planes.
\end{proposition}
\begin{proof}
Along \(\mathcal C_\sigma\) the cutoff is identically one together with all
of its derivatives.  With the convention \eqref{eq:ham-convention} and
\eqref{eq:cp2-canonical}, Hamilton's equations for
\(K_\sigma=-h_\sigma(q)\) are
\[
 \dot q=0,\qquad \dot p=\grad h_\sigma(q).
\]
Thus a point of the source graph over \(q\in\supp h_\sigma\) follows
\[
 p(\tau)=\mathrm df_\sigma(q)+\tau\,\mathrm dh_\sigma(q),
\]
which stays in the region where \(\chi_\sigma=1\) by construction.  

Next we show that the cutoff does not
alter the desired graph shear.
In the canonical coordinates, the
Hamiltonian equations for
$K_\sigma=-\chi_\sigma(q,p)h_\sigma(q)$
are
\[
\begin{aligned}
\dot q
&=-h_\sigma(q)\nabla_p\chi_\sigma(q,p),\\
\dot p
&=\chi_\sigma(q,p)\nabla h_\sigma(q)
+h_\sigma(q)\nabla_q\chi_\sigma(q,p).
\end{aligned}
\]
Fix $q_0\in\operatorname{supp}h_\sigma$.
Consider the curve
\[
\gamma_{q_0}(\tau)=
\bigl(q_0,
\mathrm df_\sigma(q_0)
+\tau\,\mathrm dh_\sigma(q_0)\bigr),
\qquad 0\le\tau\le1.
\]
By definition,
$\gamma_{q_0}([0,1])\subset\mathcal C_\sigma$.
Since $\chi_\sigma$ is identically one
on a neighborhood of $\mathcal C_\sigma$,
we have
\[
\chi_\sigma=1,\qquad
\mathrm d\chi_\sigma=0
\]
along the entire curve.
Consequently, $\gamma_{q_0}$ solves the
Hamiltonian equations for $K_\sigma$
with initial condition
$(q_0,\mathrm df_\sigma(q_0))$.
By uniqueness of solutions,
\[
\varphi_{K_\sigma}^\tau
(q_0,\mathrm df_\sigma(q_0))
=
\gamma_{q_0}(\tau).
\]
If $q_0\notin\operatorname{supp}h_\sigma$,
then $h_\sigma$ and $\mathrm dh_\sigma$
vanish near $q_0$.
Hence the Hamiltonian vector field
vanishes on the entire fiber over
$q_0$, and the corresponding point
of the source graph remains fixed.
Thus, for every $q\in E_\sigma$,
\[
\begin{aligned}
\varphi_{K_\sigma}^1
(q,\mathrm df_\sigma(q))=(q,\mathrm df_\sigma(q)
       +\mathrm dh_\sigma(q))=(q,\mathrm d\widetilde f_\sigma(q)).
\end{aligned}
\]
Therefore the image of the entire
local graph is exactly the graph of
$\mathrm d\widetilde f_\sigma$. 
Its intersections with the zero section
are precisely the critical points of
$\widetilde f_\sigma$.
By the controlled merger construction,
there is exactly one such critical point
in each disk $E_\sigma$. 
Here the critical-point count applies to all of each graph disk, not
only \(E_\sigma\): the modification equals \(f_\sigma\) near
\(\partial E_\sigma\), and the unmodified part of the disk has no
other critical point by Lemma~\ref{lem:cp2-level-arcs}.

For completeness, partition the initial source torus into points
inside the two supports of \(K_\sigma\) and points outside both
supports.  The supports are disjoint and, by the global single-sheet
property of Lemma~\ref{lem:cp2-level-arcs}, their intersections with
\(\Phi_U(T_{\mathrm{Cl}})\) consist only of the graph patches
already considered.  Every such point follows the explicitly computed
trajectory, and all its possible intersections with the zero section
have just been counted.  A point outside the supports has
\(X_K=0\) on a neighborhood, so its constant trajectory is the unique
solution of Hamilton's equation and cannot enter a support later.
It is therefore unchanged and cannot give a new intersection;
Proposition~\ref{prop:cp2-four} places all four original intersections
inside the two modification disks.  No additional intersections are
possible anywhere in \(\CP^2\).  Finally, the Hessian of the local model
\(t(t^2-s^2)\) vanishes at its unique critical point.  Lemma
\ref{lem:position-graph} then gives equality of the tangent planes.
\end{proof}

Now we choose nonnegative \(b_0,b_1\in C^\infty(\R/\mathbb Z)\), with supports
compactly contained in \((0,1/2)\) and \((1/2,1)\), respectively, and
\(\int_0^1b_0\,\mathrm dt=\int_0^1b_1\,\mathrm dt=1\).  Define
\begin{equation}\label{eq:cp2-H}
 H_{\CP^2}(t,x)=b_0(t)H_U(x)+b_1(t)K(x).
\end{equation}
It vanishes near the endpoints of the time interval and hence is smooth and
periodic.

\begin{proof}[Proof of Theorem~\ref{thm:cp2}]
The disjoint time supports in \eqref{eq:cp2-H} give
\[
 \varphi_{H_{\CP^2}}^1=\varphi_K^1\circ\Phi_U.
\]
The conclusion is therefore exactly Proposition~\ref{prop:cp2-two}.
\end{proof}

\section{Transverse perturbations}

The cubic expression permits a direct check of how the two
degenerate intersections relate to the transverse bound
\eqref{eq:AG}.

\begin{proposition}\label{prop:perturbation}
There are Hamiltonian diffeomorphisms arbitrarily close to
\(\varphi_H^1\) in the \(C^\infty\) topology whose images of
\(L\) intersect \(L\) transversely at exactly four points.
\end{proposition}
\begin{proof}
Work before the final rotation \(\mathcal R^{-1}\).
Near the unique critical point of each
\(\widehat f_{\rho,\sigma}\), use its fixed coordinates
\((s,t)\), in which the function is \(c+t(t^2-s^2)\).
Choose a smooth cutoff \(\psi\) supported in this coordinate
neighborhood and equal to one in a smaller neighborhood of
the origin. Replace the function by
\[
 \widehat f_{\rho,\sigma}+\eps\,t\psi(s,t),
 \qquad \eps>0.
\]
In the smaller neighborhood its critical equations are
\[
 -2st=0,\qquad 3t^2-s^2+\eps=0.
\]
They have exactly the two solutions
\((s,t)=(\pm\sqrt\eps,0)\).
At each, the Hessian determinant is \(-4\eps<0\).
The gradient of the unperturbed function has a positive
lower bound on the compact complement of that smaller
neighborhood within the cutoff support.
For sufficiently small \(\eps\), no new critical point
appears there. The two original critical points are thus
replaced by four nondegenerate critical points in total.

The potential change is realized by the same Hamiltonian
graph shear as in Proposition~\ref{prop:realization}, with
source-potential increment
\(-\rho^2\eps\,t\psi\), composed with the position cutoff.
For this step \(\rho>0\) and all charts are fixed.
The increments tend to zero with all derivatives as
\(\eps\to0\), so their Hamiltonian flows tend to the identity
in \(C^\infty\). The branch and support verification of
Proposition~\ref{prop:realization} applies, after enlarging
the fixed position bound if needed.
Lemma~\ref{lem:graph} proves transversality.
The final rotation transfers this conclusion to the
original torus \(L\).
\end{proof}

\end{document}